\documentclass[12pt]{article}
\pdfoutput=1
\usepackage{fullpage}
\usepackage{graphicx}
\usepackage{hyperref}
\usepackage{amsmath}
\usepackage{amssymb}
\usepackage{amsfonts}
\usepackage{amsthm}

\numberwithin{equation}{section}

\newtheorem{theorem}{Theorem}[section]
\newtheorem{corollary}[theorem]{Corollary}

\newtheorem{lemma}[theorem]{Lemma}
\newtheorem{proposition}[theorem]{Proposition}
\newtheorem{remark}[theorem]{Remark}

\title{Interpolating vector fields for near identity maps
and averaging}
\author{V.~Gelfreich$^1$ and A.~Vieiro$^2$\\[4pt]
$^1$ \small Mathematics Institute, University of Warwick, \small Coventry CV4 7AL, UK\\
\small{\tt v.gelfreich@warwick.ac.uk}\\[4pt]
$^2$ \small Departament de Matem\`atiques i Inform\`atica,\\ \small
Universitat de Barcelona, Gran Via 585, 08007 Barcelona, Spain\\
\small{\tt vieiro@maia.ub.es}
}

\begin{document}
\maketitle
\begin{abstract}
For a smooth near identity map, we introduce the notion of
an interpolating vector field  written in terms of iterates of the map.
Our construction is based on  Lagrangian interpolation 
and provides an explicit expressions for autonomous
vector fields which approximately interpolate the map.
We study properties of the interpolating vector fields and
explore their applications to the study of dynamics.
In particular, we  construct adiabatic invariants for symplectic near identity maps.
We also introduce the notion of a Poincar\'e section for a near identity map
and use it to visualise dynamics of  four dimensional maps. 
We illustrate our theory with several examples, including  
the Chirikov standard map and a symplectic map in dimension four, 
an example motivated by the theory of Arnold diffusion.
\end{abstract}


\section{Introduction} \label{Sec:Introduction}

Near identity maps naturally appear in several branches of mathematics and
mathematical physics. They are often used for numerical integration of
ordinary differential equations and for studying the  evolution under fast
time-periodic forcing. In the perturbation theory of integrable Hamiltonian
systems with three degrees of freedom,  near identity maps are used 
to describe  dynamics in a small neighbourhood of a double resonance, see e.g.
\cite{GSV2013}. The last problem is one of the central topics in the study of
the Arnold diffusion.
In these problems some important details of the discussion depend on the smoothness class of  
the map and, for the sake of simplicity, we mainly consider the 
infinitely differentiable case and, where appropriate, assume analyticity.

It is well known that a near identity map $F_{\epsilon}$, 
where $|\epsilon|  \ll   1$ and $F_0=\mathrm{I}$, is  formally embedded into an autonomous flow
\cite{Tak74}. Indeed, a formal construction can be used to find a vector field in the
form of a formal power series in $\epsilon$ such that the formal series of
the corresponding time-$\epsilon$ map coincides with the Taylor series of the original map.
Truncated  series define a vector field which interpolates the map with an
error of order of the first omitted term. 
Note that this construction is formal only
as in general the map cannot be embedded into an autonomous flow, i.e. 
generic $F_{\epsilon}$ cannot be represented as a time-$\epsilon$ map of an autonomous
vector field~\cite{Nei84}. Nevertheless, the difference between them
can be made  a flat function in $\epsilon$ at $\epsilon=0$~\cite{Sim00}.

An alternative approach to the study of near identity maps is based on
averaging (see e.g. \cite{Nei84}). A near identity map $F_{\epsilon}$ can be
written as a time-$\epsilon$ map of a time-periodic vector field called a {\em
suspension} of the map. Then the averaging procedure can be used recursively
to eliminate the time-dependence up to any order of the small parameter. 
In the analytic case, a suspension  can be found in the form of  an 
analytic autonomous vector field plus an exponentially small non-autonomous
 term \cite{Nei84} (see also \cite{BroRouSim96} for a concrete example).

Both these methods are constructive, nevertheless, none of them provides
an easy way to find an explicit expression for a vector field which
approximately interpolates the map. In this paper we propose a new method for
construction of such vector fields with the help of the Lagrange interpolation of
 iterates of the map, henceforth referred to as {\em interpolating vector fields}.
We  analyse the errors of the approximation of $F_{\epsilon}$ by the
time-$\epsilon$ map of the interpolating vector fields. 

The interpolating vector fields provide a useful tool for analysis of dynamics.
In particular, we introduce the notion of a ``Poincar\'e section'' for the map $F_\epsilon$.
Similarly to  classical Poincar\'e section, this construction
reduces the dimension of the phase space by one and, consequently, can be used to visualise trajectories of  four-dimensional maps.
Compared to the method of phase-space slices traditionally used for this purpose (see e.g. \cite{RLBK2014})
our method apparently shows higher resolution of details.

 In the case when the map $F_\epsilon$ is symplectic, 
 an interpolating vector field can be used to 
construct an adiabatic invariant $h_n$ where $n$ stands for the order of the
interpolating vector field. The function $h_n$ is approximately preserved
for many iterates of the map and can be used for a further reduction of
the dimension.

In order to illustrate the effectiveness  of the method and before discussing 
technical details involved in the construction,  
 let us consider a 4-dimensional near identity Froeschl\'e-like symplectic diffeomorphism
$T_{\epsilon}$ defined on $\mathbb T^2\times \mathbb R^2$. 
In coordinates $(\psi_1,\psi_2,J_1,J_2)$ this map is defined by the equation \eqref{4Dmap}
but the precise expression is not important for the current discussion.
Let $\Sigma$ be  a three-dimensional cylinder
defined by the equality $\psi_1=\psi_2$.  Fig.~\ref{3dPoinc}(a)  shows points
from several trajectories of $T_\epsilon$ projected along 
an interpolating vector field onto $\Sigma$.  We call  this picture  {\em a
Poincar\'e section for the map $T_\epsilon$}.

\begin{figure}[htbp]
	\begin{center}
		\includegraphics[width=14cm]{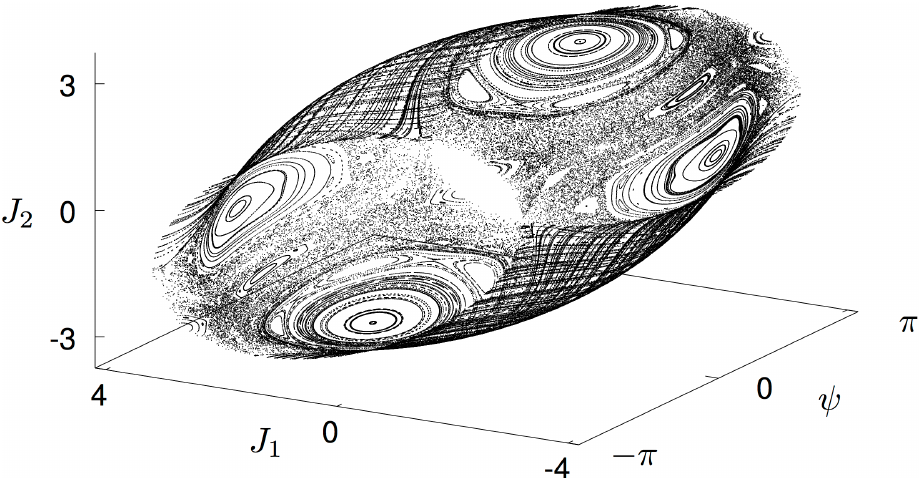} \\[10pt]
		(a)\\[25pt]
		\includegraphics[width=8cm]{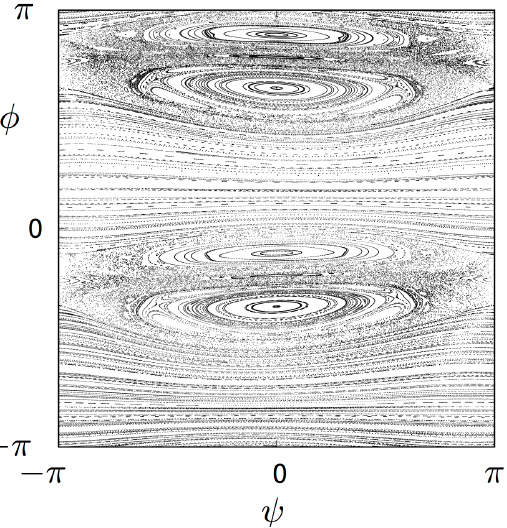}\kern 0.5cm
		\includegraphics[width=8cm]{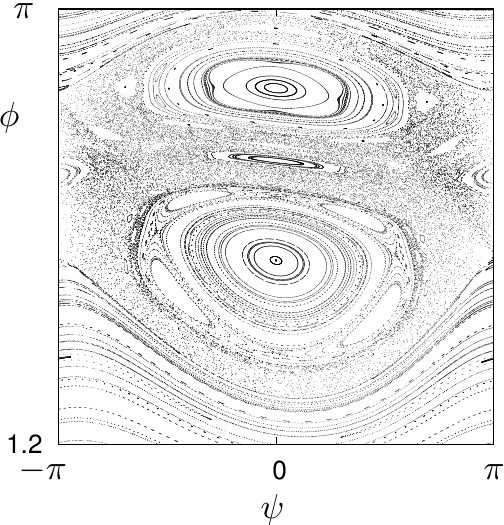}\\
		(b)\kern 7.25cm(c)
	\end{center}
	\caption{Dynamics of a symplectic map on a four-dimensional cylinder $\mathbb T^2\times\mathbb R^2$.
		The map is defined by \eqref{4Dmap} with parameter given by \eqref{parameters},  $\epsilon=0.35$,
		and 400 initial conditions are taken on $\Sigma \cap \{h_{10}=4\}$ (see Section~\ref{Se:numerics} for details of the
model and  numerical techniques used to produce these pictures).  
(a) Trajectories of the map are projected on a three-dimensional 
Poincar\'e section (a three dimensional cylinder $\mathbb T\times\mathbb R^2$)
along an interpolating vector field.
Note that all points are in a small neighbourhood of a two-dimensional torus.
(b) A two dimensional projection of  (a)
using the coordinate $\phi = \arg(J_1+ i J_2) \in
(-\pi,\pi]$ for the vertical axis. The analogy with the phase space of a
two-dimensional area-preserving map is apparent. 
(c) A magnification of (b)  where further details can be observed.}
	\label{3dPoinc}
\end{figure}

In Fig.~\ref{3dPoinc} all initial conditions are chosen from a torus defined by the intersection
of $\Sigma$ with a level set of the adiabatic invariant $h_{10}$. 
Since $h_{10}$ is approximately preserved by the map,
 all points shown on  Fig.~\ref{3dPoinc}(a)
 are in  a small neighbourhood of the  torus. Then  a
clearer image of the dynamics is obtained with the help two  angular 
coordinates, namely, $\psi$ and $\phi=\arg(J_1+i J_2)$.
  Fig.~\ref{3dPoinc}(b) shows
the  projection of  Fig.~\ref{3dPoinc}(a)   onto the torus $(\psi_1,\phi)$, and
Fig.~\ref{3dPoinc}(c) is a magnification of a strip located near the top  of Fig.~\ref{3dPoinc}(b).
 These picture
reveal that  the four-dimensional dynamics of $T_{\epsilon}$
resemble  a two-dimensional area-preserving map. This similarity
can be explained with the help of the normal form theory
\cite{Nekhoroshev1977}. The latter also provides an alternative visualisation
method (see \cite{EH13}). We stress that the method of this paper is much easier to
implement and provides higher accuracy in representation of details of the
dynamics as our pictures represent true trajectories of $T_\epsilon$ since  the
interpolation is used only for selecting initial conditions and projecting
points onto $\Sigma$.
 The details of the model and
the method are described in Section~\ref{Se:numerics}. 

\medskip

To give more formal definitions for the objects discussed in this paper, 
consider a smooth one-parameter near identity family of maps
$F_\epsilon: D\to \mathbb R^m$
where $m \geq 1$, $D\subset \mathbb R^m$ is an open domain
and $|\epsilon|<\epsilon_0$. Then the map can be written in the form
$$
F_\epsilon(x)=x+\epsilon \, G_\epsilon(x).
$$
The function $G_0$ obtained by setting $\epsilon=0$
 is often called {\em a limit vector field}.
Its time-$\epsilon$ map is $\epsilon^2$-close to~$F_\epsilon$.

The central object of this paper is  {\em an interpolating vector field for the
near identity map $F_\epsilon$} defined in the following way. Fix a
positive integer $n$.  Let $x\in D$ and suppose that the iterates $x_k=F_\epsilon^k(x)\in D$ are defined for $-n\le k\le n$.
Then there is a unique polynomial $p_n$  of degree $2n$ in $t$
 such that
\begin{equation}\label{Eq:interpol}
x_k=p_n(t_k,x_0,\epsilon)
\end{equation}
for every $t_k=\epsilon k$ with  $|k|\le n$.
The coefficients of $p_n$ depend on $x$ and $\epsilon$.
The interpolating vector field is defined as the velocity vector of
the interpolating curve at $t=0$, i.e.,
\begin{equation}\label{Eq:interpolVF}
X_n(x,\epsilon)=\partial_tp_n(0,x,\epsilon)
\end{equation}
where $\partial_tp_n$ denotes the derivative of $p_n$ with respect to  $t$.
In Section~\ref{Se:teo} we will use the Lagrange interpolation to show that
 the interpolating vector field $X_n$ is  a linear combination of the iterates of the initial point $x=x_0$,
\begin{equation}\label{Eq:symintvf}
X_n(x,\epsilon)=\epsilon^{-1}\sum_{k=1}^np_{nk}\left(x_k-x_{-k}\right),
\end{equation}
where 
\begin{equation}\label{Eq:pnk}
p_{nk}=\frac{(-1)^{k+1}(n!)^2}{k(n+k)!(n-k)!}\,.
\end{equation}
Note that the coefficients $p_{nk}$ are independent of the map $F_\epsilon$, e.g.,
 for $n=1$ we get
$$
X_1(x,\epsilon)=
\frac{
	F_\epsilon(x)-
	F_\epsilon^{-1}(x)
}{2\epsilon}.
$$
We will show that the time-$\epsilon$ map of this vector field is $|\epsilon|^3$-close to $F_\epsilon$.
Thus $X_1$ can provide a more accurate interpolation for $F_\epsilon$ than the limit flow.

If  $F_\epsilon:D\to D$ is a diffeomorphism, then
the domain $D$ is invariant and  the
equation~\eqref{Eq:symintvf} defines  $X_n$ on $D$ for every $n\in\mathbb N$.  
In general, we do not  assume that $D$ is  invariant,
and   iterates of a point $x_0\in D$ may leave the
domain. Nevertheless, it is not difficult to show that for any compact subset $D_0\subset D$,
 there is a constant $r_0>0$ such that
$x_k=F_\epsilon^k(x)$ is defined for every $x\in D_0$ and every integer $k$ 
provided $|\epsilon k|\le r_0$. Therefore $X_n$ is defined on $D_0$ for every $n\le r_0 |\epsilon|^{-1}$. 

If the map $F_\epsilon$ is a lift
of a diffeomorphism of a torus or cylinder, then the interpolating vector field is
also a lift of a vector field defined on the same manifold.  It is possible to
generalize the definition onto a general manifold 
but this discussion  is beyond the aims of the present paper.

We note that if  $F_\epsilon$ and $F_\epsilon^{-1}$ are both polynomial (e.g. the classical H\'enon map
and its generalisations), then the interpolating vector field is also polynomial.

In the symplectic case,  the interpolating vector fields remove the need of computing an
averaged Hamiltonian. This opens an opportunity of  performing massive numerical
explorations of a multi-dimensional  phase space, reducing the computation time significantly.
Possible dynamical applications of this visualizing techniques include, among
others,  studying of area-preserving maps in $\mathbb C^2$ and
of Arnold diffusion  for 4-dimensional symplectic maps near
double resonances.

The rest of this paper is organised as follows. In Section~\ref{Se:teo} we
briefly summarise the Lagrangian interpolation theory   and  derive
some properties of the interpolating vector fields.  Then we analyse the
accuracy of interpolation.  We discuss connections between the interpolating
vector field and suspensions of the map $F_\epsilon$.  In particular, we
discuss connections to averaging methods and   the optimality of choosing $n$  as a function of
$\epsilon$.  We show  that in the analytic case the interpolation error can
become exponentially small in $\epsilon$.  Finally, in subsection
\ref{Sec:adiabatic_invariant} we consider the symplectic setting.  Although the
interpolating vector field $X_n$ is not symplectic, it still can be used to
define an adiabatic invariant $h_n$ which remains approximately constant for
about  $\epsilon^{-2n}$ iterates of  the map $F_{\epsilon}$.

In  Section~\ref{Se:numerics} we present results of
our numerical experiments to illustrate the use of the interpolating vector fields
for exploring  the dynamics of maps in regions where
they are near the identity. 
First,  in Section~\ref{APM} we  consider the
2-dimensional symplectic case and  perform experiments with  the Chirikov standard map
as a model. This is a preliminary example which 
illustrates various useful features   of the interpolating vector fields.  Moreover, one can
compare the interpolating vector field directly with the map as the two dimensional dynamics
is easy to visualize.

The main advantage of interpolating vector fields is that they provide a useful
tool to investigate higher dimensional phase spaces. In particular, they provide
a powerful tool for computing projections of the discrete dynamics onto
subspaces of codimension-one. We use this tool to define 
Poincar\'e sections for maps, an extension of the classical construction
traditionally restricted to systems with continuous time only,
see details in Section~\ref{Poincare}.
We illustrate this construction for 
the 4-dimensional symplectic map $T_{\epsilon}$ to obtain representations of
different types of its dynamics with the help of plots in dimensions two and three.

\section{Interpolating vector fields \label{Se:teo}}

\subsection{Lagrange interpolating polynomials}

An interpolating polynomial can be explicitly written with the help of Lagrange
interpolating polynomials (see e.g. \cite{textbook}). Suppose that  nodes are located at integer points
symmetrically with respect to $\tau=0$.  Then Lagrange's fundamental
interpolating polynomial is
\begin{equation}
\pi_n(\tau)=\prod_{k=-n}^n(\tau-k)
\end{equation}
and Lagrange's basis polynomials are
\begin{equation}\label{Eq:lagrpol}
\pi_{nk}(\tau)=\frac{\pi_n(\tau)}{\pi_n'(k)(\tau-k)}.
\end{equation}
For every $n\in\mathbb N$ and $|k|\le n$  the equation~\eqref{Eq:lagrpol} defines the unique polynomial of
degree $2n$ such that $\pi_{nk}(j)=\delta_{kj}$ for every $|j|\le n$, where
$\delta_{kj}$ is the Kronecker delta. 
For $1\le |k|\le n$ we differentiate the
equation~\eqref{Eq:lagrpol} and, taking into account that $\pi_n(0)=0$, we get 
$
\pi_{nk}'(0)=-\frac{\pi_n'(0)}{k\pi_n'(k)}.
$
Since $\pi_n'(0)=(-1)^n(n!)^2$ and
\begin{equation}\label{Eq:pi_n'(k)}
\pi_n'(k)=\prod_{j=-n}^{k-1}(k-j)\prod^{n}_{j=k+1}(k-j)=(-1)^{n-k}(n+k)!(n-k)!
\end{equation}
we conclude  that  $\pi_{nk}'(0)=p_{nk}$ where 
\begin{equation}\label{Eq:pnk1}
p_{nk}=\frac{(-1)^{k+1}(n!)^2}{k(n+k)!(n-k)!}\,.
\end{equation}
 Since $\pi_n$ is odd, $\pi_{n0}$ is even
and, consequently,   $p_{n0}=\pi_{n0}'(0)=0$. The constants $p_{nk}$ are used in
the equation  \eqref{Eq:symintvf} which defines the interpolating vector fields.
Later we will need some properties of these
constants.

\begin{proposition}
	For every $n$ fixed, $(-1)^{k+1}p_{nk}$ is a monotone decreasing positive sequence  for $1\le k \le n$,
	and $p_{n,{-}k}=-p_{nk}$.
	Moreover, 
\begin{equation}
\label{Eq:sumpnk}
\sum_{k=-n}^n p_{nk}k^j=
\begin{cases}
1,&
\mbox{if $j=1$},\\
0,&\mbox{if $j=0$ or $2\le j\le 2n$}
,
\end{cases}
\end{equation}
\begin{equation}\label{Eq:|pnk|}
\sum_{k=1}^n|p_{nk}|= \frac{H_n}{2}\qquad\mbox{and}\qquad
\sum_{k=1}^n p_{nk}= H_{2n}-H_{n},
\end{equation}
where  $H_n=\sum_{k=1}^n k^{-1}$ is the $n$-th harmonic number.
\end{proposition}

\begin{proof}
The first  equality can be derived from the following observation.
Let $0\le j\le 2n$. Interpolating any polynomial of degree $j$ with 
the help of a polynomial of degree $2n$, we recover the original polynomial.
Applying this rule to $\tau^j$ we get that 
$$
\sum_{k=-n}^n \pi_{nk}(\tau)k^j=\tau^j.
$$
Differentiating this equality with respect to $\tau$ at $\tau=0$ we obtain~\eqref{Eq:sumpnk}.
\end{proof}

Let $r>0$ and consider a continuous curve $\gamma:[-r,r]\to \mathbb R^m$.
For any $\epsilon\ne0$ and $n\in\mathbb N$ such that $n|\epsilon|<r$,
the equation
$$
p_n(t,\epsilon)=\sum_{k=-n}^n\pi_{nk}(\epsilon^{-1}t)\gamma(k \epsilon)
$$
defines the  unique  polynomial   of degree $2n$
such that $p_n(k\epsilon,\epsilon)=\gamma(k \epsilon)$ for $|k|\le n$.
Let $v_{n}(\gamma,\epsilon)=\partial_tp_n(0,\epsilon)$. Taking into account the expression for $p_n$
we get
\begin{equation}\label{Eq:vn}
v_{n}(\gamma,\epsilon)=
\epsilon^{-1}\sum_{k=-n}^np_{nk}\gamma(k \epsilon ).
\end{equation}
Note that  although this definition contains a division by $\epsilon$,
 the following lemma implies that $v_n$ can be extended to $\epsilon=0$ by continuity
 provided the curve is sufficiently smooth.

\begin{lemma}\label{Pro:gamma}
	If $\gamma\in \mathcal{C}^{2n+1}\bigl([-r,r]\bigr)$ 
	and $|\epsilon|n<r$, then
	\begin{equation}
\left|	\dot{\gamma}(0) -v_{n}(\gamma,\epsilon)\right|
\le \frac{\epsilon^{2n}(n!)^2}{(2n+1)!} \max_{|t|\le r}\left|\gamma^{(2n+1)}(t)\right|.
	\end{equation}
\end{lemma}

\begin{proof}
    The standard theory of interpolating polynomials, see e.g.
    \cite{textbook}, states that
	$$
	\gamma(t)=p_n(t,\epsilon)+ \frac{\epsilon^{2n+1}\pi_n(\epsilon^{-1}t)}{(2n+1)!} w_n(t,\epsilon),
	$$
	where 	the function $w_n$ can be expressed  as $ w_n(t,\epsilon)=\gamma^{(2n+1)}(\tau_{*}(t,\epsilon))$ with
	$\tau_{*}(t,\epsilon)\in[-n\epsilon,n\epsilon]$,   an intermediate value of the time variable.
	Since the zeroes of $\pi$ are all simple, $w$ is continuous at $t=k\epsilon$.
	Differentiating  with respect to $t$ at $t=0$ and taking into account that  $\pi_n(0)=0$,
we get  
	$$
\dot\gamma(0)=v_{n}(\gamma,\epsilon)+
\frac{\epsilon^{2n} 	\pi_n'(0)}{(2n+1)!}
\gamma^{(2n+1)}(\tau_{*}(0,\epsilon))
	.
	$$
The estimate of the lemma immediately follows from \eqref{Eq:pi_n'(k)}. 
This inequality implies that $v_{n}(\gamma,\epsilon)\to\dot\gamma(0)$ when $\epsilon\to0$.
\end{proof}

\begin{lemma}[stability with respect to perturbations]\label{Pro:stability}
	If $\gamma_1,\gamma_2\in \mathcal{C}^0([-r,r])$,  $\epsilon\ne0$  and $n|\epsilon|<r$, then
	$$
	\left| 
	v_n(\gamma_2,\epsilon)-v_n(\gamma_1,\epsilon)
	\right|
	\le |\epsilon|^{-1}H_n\left\|\gamma_2-\gamma_1\right\|
	$$
	where $H_n= \sum_{k=1}^n k^{-1}$ is the $n$-th harmonic number.
\end{lemma}

\begin{proof}
	Using \eqref{Eq:vn} and the triangle inequality we get
	\begin{eqnarray*}
	\left|  v_n(\gamma_2,\epsilon)-v_n(\gamma_1,\epsilon)\right|
&\le&
|\epsilon|^{-1}\sum_{k=-n}^n|p_{nk}|\left| \gamma_2(k\epsilon)-\gamma_1(k\epsilon))\right|
\\
&\le&
|\epsilon|^{-1}\left\|\gamma_2-\gamma_1\right\|\sum_{k=-n}^n|p_{nk}|.
	\end{eqnarray*}
Then the first equality in \eqref{Eq:|pnk|} implies the estimate of the lemma.
\end{proof}

\subsection{Basic properties of interpolating vector fields}

The interpolating polynomial of the equation~\eqref{Eq:interpol} can be explicitly written 
in the form of a Lagrange interpolating polynomial, namely,
$$
p_n(t,x,\epsilon)=\sum_{k=-n}^n\pi_{nk}(\epsilon^{-1} t)x_k, \quad \text{ where }  x_k=F_{\epsilon}^k(x).
$$
Differentiating this equality with respect to $t$ at $t=0$ and using that $\pi_{nk}'(0)=p_{nk}$, 
we recover the explicit expression~\eqref{Eq:symintvf} for the interpolating vector field~\eqref{Eq:interpolVF}:
\begin{equation}\label{Eq:interpol_vf}
X_n(x,\epsilon)=\epsilon^{-1}\sum_{k=-n}^np_{nk}x_k
=\epsilon^{-1}\sum_{k=1}^np_{nk}(x_k-x_{-k})
.
\end{equation}
 For example, for $n=2$ we get
$$
X_2(x,\epsilon)=
\frac{2
}{3\epsilon}
\left(	F_\epsilon(x)-
F_\epsilon^{-1}(x)
\right)
-
\frac{1
}{12\epsilon}
\left(	F_\epsilon^2(x)-
F_\epsilon^{-2}(x)
\right).
$$
The definition of the interpolating vector field involves division by $\epsilon$,
but we will see that $X_n$ can be continuously extended to $\epsilon=0$.

\begin{proposition}
		Let $D_0 \subset D$ be compact. 
		Then there is $r_0>0$ such that the interpolating vector field is
                defined for all $x\in D_0$ and all $n\in\mathbb N$ such that $n|\epsilon|\le r_0$.
		For every $n$, the interpolating vector field $X_n$ is as
                smooth as the map $F_\epsilon$ itself and
$
X_n(x,0)=G_0(x).
$
\end{proposition}

\begin{proof}
Let $F_\epsilon(x)=x+\epsilon G_{\epsilon}(x)=x+\epsilon g_1(x,\epsilon)$. Obviously, $g_1(x,0)=G_0(x)$. Then  the finite induction
implies that 
\begin{equation}
	\label{Eq:xk}
	x_k=F_\epsilon^k(x)=x_0+\epsilon\sum_{j=0}^{k-1} g_1(x_j,\epsilon)
\end{equation}
for every $k\in\mathbb N$ such that $x_j\in D$ for all $0\le j\le k$.
Similarly, the inverse function theorem implies that
$F^{-1}_\epsilon(x)=x-\epsilon g_{-1}(x,\epsilon)$, where
$g_{-1}(x,0)=-G_0(x)$. Repeating the previous arguments we conclude that
\begin{equation}
	\label{Eq:x-k}
	x_{-k}=F_\epsilon^{-k}=x_0+\epsilon\sum_{j=0}^{k-1} g_{-1}(x_{-j},\epsilon).
\end{equation}
Using the finite induction, the equations  \eqref{Eq:xk} and \eqref{Eq:x-k} and
compactness of $D_0$, we prove that there is $r_0>0$ such that
$F^k_\epsilon(x)\in D$ for all $x\in D_0$ provided  $|k\epsilon|<r_0$. 

Substituting the expressions \eqref{Eq:xk} and \eqref{Eq:x-k} into the
definition \eqref{Eq:interpol_vf} of $X_n$, and using \eqref{Eq:sumpnk} with $j=0$, we obtain
\begin{equation}\label{Eq:XnGe}
	X_n(x,\epsilon)=
	\sum_{k=-n}^{n}p_{nk}\sum_{j=0}^{|k|-1} g_{\mathrm{sign}( k)}(x_{\mathrm{sign}(k) j},\epsilon) 
\end{equation}
and  the smoothness of $X_n$ follows directly. Finally, we note that if $\epsilon=0$, then $x_j=x_0=x$	for every $j$.  
Substituting 
$\epsilon=0$ into the equation \eqref{Eq:XnGe} we get
$$
  X_n(x,0)= \sum_{k=-n}^{n}p_{nk}\sum_{j=0}^{|k|-1} g_{\mathrm{sign}( k)}(x,0)= \sum_{k=-n}^{n}p_{nk}kG_0(x) .
$$
Then the equation \eqref{Eq:sumpnk} with $j=1$ implies that $X_n(x,0)=G_0(X)$.
\end{proof}	

An interpolating vector field can be used to approximately restore a vector field from its time-$\epsilon$ map.

\begin{proposition}\label{Eq:resoreY}
	Let $\Phi^\tau$ be a time-$\tau$ flow of a smooth autonomous vector field~$Y$
	and $n\in\mathbb N$. Let $X_n$ be an interpolating vector field for the map $\Phi^\epsilon$.
	Then
	$$
	Y(x)=X_n(x,\epsilon)+O(\epsilon^{2n}).
	$$
	The estimate is uniform on every compact subset of $D$.
\end{proposition}

\begin{proof} 
	Let $x\in D$ 
	and $\gamma$ be a solution of the initial value problem for the ordinary differential equation 
	$\dot \gamma=Y(\gamma)$, $\gamma(0)=x$. 
	Then Proposition~\ref{Eq:resoreY} follows from Lemma~\ref{Pro:gamma} as the solution is smooth.
\end{proof}
	
We note that the following arguments can  also be used to prove the
proposition.  Expanding the solution $\gamma$ of the initial value problem into
Taylor series in time we get
$$
\gamma(t)=x+\sum_{k=1}^{2n}\frac{t^k}{k!}\gamma^{(k)}(0)+r_{2n}(t)
$$
where $\gamma^{(1)}(0)=Y(x)$ and $r_{2n}(t)=O(t^{2n+1})$.
Substituting this expression into the definition of the 
interpolating vector field we get
\begin{eqnarray*}
\epsilon	X_n(x,\epsilon)&=&\sum_{k=-n}^np_{nk}\gamma(\epsilon k)
	=
	\sum_{k=-n}^np_{nk}\left(\sum_{j=1}^{2n}\frac{\epsilon^j k^j}{j!}\gamma^{(j)}(0)+r_{2n}(k\epsilon)
	\right)
\\
	&=&\sum_{j=1}^{2n}\frac{\epsilon^j }{j!}\gamma^{(j)}(0)\sum_{k=-n}^nk^jp_{nk}+R_n(\epsilon)
=\epsilon  Y(x)+R_n(\epsilon),
\end{eqnarray*}
where $R_n(\epsilon)=\sum_{k=-n}^{n}p_{nk}r_{2n}(k\epsilon)=O(|\epsilon|^{2n+1})$.
Here we used the equation \eqref{Eq:sumpnk} to simplify the sum.

\begin{remark}
If $x_0$ is a fixed point of $F_\epsilon$, then $x_0$ is  an equilibrium of $X_n$.
\end{remark}

\begin{remark}
For a fixed $n$ and a small $\epsilon$ the interpolating vector field
$X_n(x,\epsilon)$ only involves values of $F_\epsilon$ from a small neighbourhood of
the point $x$.  
\end{remark}

\begin{remark}
If $F_\epsilon$ is a lift of a map defined on a cylinder (or
a torus), then the interpolating vector field $X_n$ is also a lift of a vector
field defined on the same manifold. This argument also applies to every
manifold obtained by factorising $\mathbb R^m$ with respect to
action of a discrete group of  (affine) linear transformations.
\end{remark}

We recall that a map $F_\epsilon$ is called {\em reversible}, if there is an  involution $R$
(i.e., $R\circ R=\mathrm{I}$) which conjugates the map and its inverse, i.e.,
$F_\epsilon^{-1}=R\circ F_\epsilon \circ R$. The involution $R$ is called a
{\em reversor}. In this paper we consider only the case when $R$ is a linear
map. This case is often used in applications.

\begin{proposition}
If $F_\epsilon$ is a reversible map with a linear reversor $R$, 
then the interpolating vector field is also reversible, 
i.e.~the application of the map $R$ changes the direction of time:
$$X_n(Rx)=-RX_n(x).$$
\end{proposition}

This proposition is proved by a straightforward computation using the symmetric
expression for the interpolating vector field \eqref{Eq:symintvf} and the fact
that the reversor $R$ maps a trajectory of $F_\epsilon$ into a trajectory of
$F_\epsilon^{-1}$.

\medskip

Let $F_\epsilon$ be a map defined on an one-dimensional interval. Suppose that
$x_0 \in \mathbb{R}$ is a fixed point of $F_\epsilon$ and let
$\lambda_\epsilon=F'_{\epsilon}(x_0)$. Then $x_0$ is an equilibrium for the
interpolating vector field, i.e. $X_n(x_0)=0$, and
$$
X_n'(x_0)=\epsilon^{-1}\sum_{k=1}^n p_{nk}\left(\lambda_\epsilon^k-\lambda_\epsilon^{-k }\right).
$$
If $\epsilon$ is sufficiently small,  the linear stability type of the fixed point is the same for the
map and for the interpolating vector field.  Moreover, 
since $F_\epsilon$ is a monotone function, it is easy to see that
the time-$\epsilon$ map of the interpolating vector field is topologically
conjugate to $F_\epsilon$. On the other hand, in general, the conjugating
function cannot be differentiable as the multipliers $\lambda_{\epsilon}$ and
$\exp(\epsilon X_n'(x_0))$ of the fixed points can be different. The
conclusions about topological equivalence do not require hyperbolicity of
$F_\epsilon$ and do not generalize to higher dimensions.

\subsection{Suspensions of a map and averaging}

In general it is not possible to construct an autonomous vector field such that
its time-$\epsilon$ map coincides with $F_\epsilon$. 
On the other hand,  it is possible to construct a time-periodic vector 
field with this property. Such vector field is called a {\em suspension of $F_\epsilon$}.
Suppose that $Y$ is a suspension of $F_\epsilon$. This vector field has the following
characteristic property:
if  $\xi=\xi(\tau,x,\epsilon)$ is a solution of the initial value problem
\begin{equation}\label{Eq:suspivp}
\partial_\tau\xi=\epsilon Y(\tau,\xi,\epsilon),
\qquad\xi(0,x,\epsilon)=x\,,
\end{equation}
then $\xi(1,x,\epsilon)=F_\epsilon(x)$. Note that we use the fast time
$\tau=\epsilon^{-1}t$ instead of the slow time $t$.

It is not too difficult to construct a suspension. Indeed, let
$\chi:[0,1]\to[0,1]$ be a monotone smooth ($\mathcal{C}^\infty$) function,
such that $\chi(0)=0$, $\chi(1)=1$ and $\chi^{(k)}(0)=\chi^{(k)}(1)=0$ for all
$k\in\mathbb N$. Then consider a curve parametrized by the function
\begin{equation}\label{Eq:intsol}
\xi(\tau,x,\epsilon)=x+\epsilon\chi(\tau ) G_\epsilon(x)
\end{equation}
with $\tau\in[0,1]$. This curve  connects the points $x$ and
$F_\epsilon(x)=x+\epsilon G_\epsilon(x)$.
Consider the map $\Phi_\epsilon^\tau : x\mapsto \xi(\tau,x,\epsilon)$. This map
is $\epsilon$-close to the identity in the $\mathcal{C}^1$ topology and, hence, a
local diffeomorphism by the inverse function theorem. Let
$\Psi_\epsilon^\tau=\left(\Phi_\epsilon^\tau\right)^{-1}$ and consider a non-autonomous vector
field 
\begin{equation}\label{Eq:Ysusp}
Y(\tau,x,\epsilon)=\chi'(\tau) G_\epsilon(\Psi_\epsilon^{\tau}(x))
\end{equation}
originally defined for $\tau\in[0,1]$ and extended periodically as a function of $\tau$.
 It is easy to see that $Y$ is smooth.

Obviously, the function \eqref{Eq:intsol} is a solution of the initial value problem
\eqref{Eq:suspivp} with the vector-field defined by \eqref{Eq:Ysusp}.
Consequently, the time-$\tau$ map of the vector field $\epsilon Y$ coincides
with $\Phi^\tau_\epsilon$ for $\tau\in[0,1]$.  Since
$\Phi_\epsilon^1=F_\epsilon$, the vector field $ Y$ is a suspension of the
map~$F_\epsilon$. 

Note that in this construction  solutions of the differential
equations~\eqref{Eq:suspivp} are given explicitly for $\tau\in[0,1]$. These
solutions extend recursively to other values of $\tau$ using the identity
$\xi(\tau,x,\epsilon)=\xi(\tau-1,F_\epsilon(x),\epsilon)$.  Since the function
$\chi$ is flat at $\tau=0$ and $\tau=1$, the function $\xi$ is smooth in time
as long as the solution remains in $D$. It is also possible to use
the flow $\Phi^\tau_\epsilon(x)=F_{\epsilon\chi(\tau)}(x)$ instead of (\ref{Eq:intsol}) in the construction
of the suspension.

\medskip

A suspension of $F_\epsilon$ is not unique and we are interested in finding a
suspension which is as close to an autonomous vector field as possible. This
task can be performed with the help of averaging~\cite{BogMit1961}. An averaging  step  consists
of a substitution of the form
\begin{equation}\label{Eq:changeS}
\xi=\xi_1+S(\tau,\xi_1,\epsilon)
\end{equation}
where the function $S$ is  periodic in $\tau$.  Substituting \eqref{Eq:changeS}
into  the differential equation  \eqref{Eq:suspivp} we get
\begin{equation}\label{Eq:susp_tmp}
(1+\partial_\xi S)\partial_\tau\xi_1+\partial_\tau S=\epsilon Y(\tau,\xi_1+S,\epsilon)
\end{equation}
where  $S$ is evaluated at $(\tau,\xi_1,\epsilon)$.
Let $\langle Y\rangle=\int_0^1 Yd\tau$ be the mean value of $Y$ over its period in $\tau$,
and $\tilde Y=Y-\langle Y\rangle$ be its oscillatory part with zero mean.
It is convenient to let
$$
S(\tau,\xi,\epsilon)=\epsilon \int_0^\tau \tilde Y(t,\xi,\epsilon)\,dt.
$$
This function is periodic in $\tau$ and 
 $S(k,\xi,\epsilon)=0$ for every $k\in\mathbb Z$. Then   the map
$\xi_1\mapsto \xi$ is $\epsilon$-close to the identity for every $\tau$ and  is exactly the identity
for $\tau\in\mathbb Z$. Therefore,   the change of variables
\eqref{Eq:changeS}
transforms $Y$  into another suspension of the same map $F_\epsilon$.
Multiplying the equation \eqref{Eq:susp_tmp} by the matrix $(1+\partial_\xi S)^{-1}$,   we write
the new suspension in  the form
$$
\partial_\tau\xi_1=\epsilon Y_1(\tau,\xi_1,\epsilon),
$$
where
$$
Y_1(\tau,\xi_1,\epsilon)=(1+\partial_\xi S)^{-1}\bigl( 
Y(\tau,\xi_1+S,\epsilon)
-
 \tilde Y(\tau,\xi_1,\epsilon)
\bigr).
$$
Taking into account that  $S=O(|\epsilon|)$, the definition of $\tilde Y$
and differentiability of $Y$ and $S$, we can check that 
 the oscillatory part of $Y_1$ vanishes at the leading order, i.e.,
$\|\tilde Y_1\|=O(\| \epsilon \tilde Y\|)$.
The averaging procedure can be repeated to further decrease the time-dependent part of the
suspension. Note that the expression for the vector field $Y_1$ contains a first derivative with respect to the space
variable, thus if $Y\in \mathcal{C}^k$ then we can claim that $Y_1\in \mathcal{C}^{k-1}$ only and the maximum number of averaging
steps is bounded by the smoothness class of~$Y$.
If $Y$ is infinitely differentiable,  the time-dependent part of the suspension vector field can be made
smaller than any power of $|\epsilon|$ by repeating the averaging procedure multiple times:
 after $n$ steps we obtain a suspension of the form
\begin{equation}\label{Eq:Yn}
Y_n(\tau,x,\epsilon)=A_n(x,\epsilon)+\epsilon ^n B_n(\tau,x,\epsilon).
\end{equation}
Neishtadt~\cite{Nei84} proved that if $Y$ is analytic in a complex neighbourhood of $D$
then after $n\sim |\epsilon|^{-1}$ averaging steps the time-dependent part of the suspension vector field
becomes exponentially small compared to $|\epsilon|$. I.e., if
$F_\epsilon$ is a family of maps analytic in a complex neighbourhood of $D$,
then there is a suspension vector field defined on $D$ such that
\begin{equation}\label{Eq:Neishtadt}
Y(\tau,x,\epsilon)=A(x,\epsilon)+B(\tau,x,\epsilon),
\end{equation}
with $B=O(\exp(-c/|\epsilon|))$ for some $c>0$.

\subsection{Suspensions of a map and  interpolating vector fields}

Although the averaging is an effective tool for  studying near identity maps,
it is not usually possible to find an explicit expression 
for the vector fields $Y_n$.

\begin{theorem}\label{Thm:AnXn}
Let $F_\epsilon:D\to\mathbb R^m$ be a smooth family of  near identity maps defined on
a domain
$D\subset \mathbb{R}^m$ with $m\geq 1$ and let $n\in\mathbb N$.  If a suspension of $F_\epsilon$ can be
written in the form
\begin{equation}
\label{Eq:Yn1}
Y(t,x,\epsilon)=A_n(x,\epsilon)+\epsilon^{2n} B_n(t,x,\epsilon)
\end{equation}
where the $\mathcal{C}^{2n}$ norms of $A_n$ and $B_n$ are bounded uniformly
with respect $\epsilon$, then for every compact $D_0\subset D$ there is a
constant $C_n$ such that
$$
\sup_{x\in D}	\left| A_n(x,\epsilon)-X_n(x,\epsilon)\right| \le C_n \epsilon^{2n} 
$$
where $X_n$ is the interpolating vector field for the map $F_\epsilon$.
\end{theorem}

\begin{proof}
	Since $D_0$ is compact, there is $\epsilon_0>0$ such that
every solution of
\begin{equation}\label{Eq:suspension}
\partial_\tau\xi(\tau,x,\epsilon)=\epsilon Y(\tau,\xi(\tau,x,\epsilon),\epsilon),
\qquad\xi(0,x,\epsilon)=x
\end{equation}
with  initial condition $x\in D_0$  remains inside $D$ for $|\tau|\le n$ and $|\epsilon|\le\epsilon_0$.
Since $Y$ is a suspension of $F_\epsilon$, we have $F_\epsilon^k(x)=\xi(k,x,\epsilon)$ for $|k|\le n$.
Repeatedly differentiating the equation \eqref{Eq:suspension} with respect to $\tau$ and taking into account the form of $Y$,
we see that
there is a constant $C$ such that
 $$
 \left|\partial _\tau^k\xi(\tau,x,\epsilon)\right|\le C |\epsilon|^k  
 $$ 
 for $k\le 2n+1$.
Then the theorem follows from the  Lemma~\ref{Pro:gamma} with $\gamma(t)=\xi(\epsilon^{-1} t,x,\epsilon)$.
\end{proof}

This theorem has several important corollaries.	The first one states that the time-$\epsilon$ 
shift along trajectories of the interpolating vector field approximates $F_\epsilon$.

\begin{corollary} \label{Prop:vector_field}
	If $F_\epsilon\in \mathcal{C}^{2n+1}$ and $D_0$ is a compact subset of $D$ then on this subset the interpolating vector field
	is uniformly bounded for $|\epsilon|<\epsilon_0$ and
	$$
	F_\epsilon(x)=\Phi^\epsilon_{X_n}(x)+O(|\epsilon|^{2n+1})	
	$$ 
	where $\Phi^\epsilon_{X_n}$ is the time-$\epsilon$ map associated with the vector field $X_n$.
\end{corollary}

\bigskip

If $F_\epsilon$ is analytic in a complex neighbourhood of $D_0$, then we
can choose $n\sim |\epsilon|^{-1}$ in order to obtain a  vector field  which
interpolates $F_\epsilon$ with an error exponentially small in  $\epsilon$.
To prove this, let $Y$ be the suspension of $F_\epsilon$ given by
\eqref{Eq:Neishtadt}. So its non-autonomous part $B$ is exponentially small in
$\epsilon$. Suppose that a solution of the autonomous equation
$\partial_t\eta(t,x,\epsilon)=A(\eta,\epsilon)$ with $\eta(0,x,\epsilon)=x\in D_0$ is analytic
in time for $|t|<3r$ because it remains inside the domain of the vector field $A$ for these values
of $t$.	Then the Cauchy estimate implies that $|\partial^{k+1}_t\eta|\le k!
r^{-k}\|A\|$ for $|t|\le r$.  Let $\tilde X_n(x,\epsilon)$ be the interpolating
vector field for the time-$\epsilon$ map of the flow of $A$.
Applying   Lemma~\ref{Pro:gamma} to the curve
$\gamma(t)=\eta(t,x,\epsilon)$,
 we obtain that for $n|\epsilon|<r$ 
$$
\left|A(x,\epsilon)-\tilde X_n(x,\epsilon)\right|\le \frac{\epsilon^{2n}(n!)^2}{2n+1}r^{-2n}\|A\|.
$$
Then  Stirling's formula implies
$$
\left|A(x,\epsilon)-\tilde X_n(x,\epsilon)\right|\le \left(\frac{\epsilon n }{re } \right)^{2n} \frac{4\pi n}{2n+1}\|A\|.
$$	
Since $B$ is small, the time-$t$ map of $A$ is close to the time-$t$ map of $Y$. In particular,
Lemma~\ref{Pro:stability} implies that 
$$
\left|X_n(x,\epsilon)-\tilde X_n(x,\epsilon)\right|\le  H_n |\epsilon|^{-1} C _A \|B\|.
$$
where $C_A>0$ is a suitable constant so that $C_A \|B\|$ bounds the distance
between the solutions of the initial value problems related to the vector
fields $A$ and $Y$ for $|t| \leq r$. Using the triangle inequality we get
$$
\left|A(x,\epsilon)-X_n(x,\epsilon)\right|
\le  \left(\frac{\epsilon n }{re } \right)^{2n} \frac{4\pi n}{2n+1}\|A\| + H_n |\epsilon|^{-1}C _A  \|B\|.
$$
Since $B = \mathcal{O}(e^{-c/|\epsilon|})$ for a suitable $c>0$ it follows that, taking $n=[r/|\epsilon|]$, one obtains
$$
\left|A(x,\epsilon)-X_n(x,\epsilon)\right|\le 2\pi \|A\| e^{-r/|\epsilon|} + |\epsilon|^{-1}\log |\epsilon|^{-1} C e^{-c/|\epsilon|}.
$$
Thus for this $n=n(\epsilon)$ the  vector field $X_n$ interpolates the map
with an error exponentially small  in $\epsilon$.

\subsection{Adiabatic invariants of  a symplectic map} \label{Sec:adiabatic_invariant}

Let  $m=2d$ and $\omega$ be a symplectic form on~$\mathbb R^{2d}$.
We will mainly consider the standard symplectic form
$$
\omega=\sum_{i=1}^d dx_i\wedge dx_{i+d}.
$$
A map is called {\em symplectic\/} if it preserves  $\omega$.
Even when  $F_\epsilon$ is symplectic, the corresponding interpolating vector field $X_n$ does not
necessary define a symplectic flow and, consequently, is not necessarily Hamiltonian. 
Therefore there is no reason to expect that $X_n$ has an integral.
Nevertheless, $X_n$  is very close to a symplectic one and 
can be used to define an adiabatic invariant  -- a function which is
approximately preserved on time scales much longer than~$|\epsilon|^{-1}$.

In order to construct the adiabatic invariant, we 
consider a differential one-form
$
\nu_n=\omega(X_n,\cdot).
$
In the case of the standard symplectic form
\begin{equation}\label{Eq:nustandard}
\nu_n=\sum_{i=1}^d\left(
X_n^i dx_{i+d}-X_n^{i+d}dx_{i}
\right),
\end{equation}
where $X_n^i$ denotes a component  of the vector  $X_n$.
Then we fix   a  base point $x_b\in D$ and, for every $x\in D$, consider
 an integral
\begin{equation}
	\label{Eq:hn}
h_n(x,\epsilon)=\int_{\gamma(x_b,x)}\nu_n
\end{equation}
where   $\gamma(x_b,x)$ is a curve which connects the base point $x_b$ and $x$.
If the form $\nu_n$ is closed and the domain $D$ is simply connected, this integral 
does not depend on the choice of the path (yet it depends on the choice of $x_b$) 
and uniquely defines the function $h_n:D\to\mathbb R$ such that  $dh_n=\nu_n$.

In general, the form $\nu_n$ is not necessarily closed and the integral may depend on the choice of the path.
In order to obtain a well-defined function we  choose a rule for selecting  $\gamma(x_b,x)$.
For example, if the domain $D$ is star-shaped (e.g. convex),  then
$\gamma$ can be a straight segment connecting its end points.
Another convenient choice is a path which consists of straight
segments parallel to coordinate axes. 

In contrast with the interpolating vector fields, 
a  suspension vector field  \eqref{Eq:Yn1} 
for symplectic map can be chosen (locally) Hamiltonian~\cite{Nei84,G2002}.

\begin{proposition}\label{Pro:hn}
Let $F_\epsilon$ be a smooth family of  near identity symplectic maps defined on
a domain
$D\subset \mathbb{R}^{2d}$ with $d\geq 1$ and let $n\in\mathbb N$.  
Let  $C>0$  be a constant and $\gamma(x_b,x)$ be piecewise smooth paths
such that $|\gamma(x_b,x)|\le C |x-x_b|$ for every $x\in D$.
If a suspension of $F_\epsilon$ can be written in the form of a Hamiltonian
vector field with a Hamiltonian function 
$$
h(t,x,\epsilon)=h^a_n(x,\epsilon)+\epsilon^{2n} h^b_n(t,x,\epsilon)
$$
where the $\mathcal{C}^{2n+1}$ norms of $h^a_n$ and $h^b_n$ are
bounded uniformly with respect $\epsilon$ and  $h_n^a(x_b,\epsilon)=0$,
then for every compact $D_0\subset D$ there is a constant $C_n$ such that
$$
\sup_{x\in D_0}	\left|h_n(x,\epsilon)-h_n^a(x,\epsilon)\right| \le C_n \epsilon^{2n}
$$
where $h_n$ is defined by \eqref{Eq:hn}.
 \end{proposition}

\begin{proof}
Let the path $\gamma(x_b,x)$ be  parametrized by a function $\gamma(t,x_b,x)$ with  $t\in[0,1]$,  $\gamma(0,x_b,x)=x_b$ and $\gamma(1,x_b,x)=x$.
Then  \eqref{Eq:hn} takes the form
$$
	h_n(x,\epsilon)=\int_0^1\omega\bigl(X_n(\gamma(t,x_b,x),\epsilon),\partial_t \gamma(t,x_b,x)\bigr)dt.
$$
Let $A_n$ be the Hamiltonian vector field defined by the Hamiltonian function $h_n^a$.
Then
$$
h_n^a(x,\epsilon)=\int_0^1\omega\bigl(A_n(\gamma(t,x_b,x),\epsilon),\partial_t \gamma(t,x_b,x)\bigr)dt.
$$
Using the bi-linearity of the symplectic form $\omega$, we obtain
$$
	h_n(x,\epsilon)-h_n^a(x,\epsilon)=\int_0^1\omega\bigl(X_n(\gamma(t,x_b,x),\epsilon)
	-A_n(\gamma(t,x_b,x),\epsilon),\partial_t \gamma(t,x_b,x)\bigr)dt.
$$	
Since $|\omega(v_1,v_2)|\le C_\omega |v_1|\,|v_2|$ for any two vectors $v_1,v_2$ (e.g. $C_\omega=1$ for the standard symplectic form), 
we get 
$$
|h_n(x,\epsilon)-h_n^a(x,\epsilon)|\le C_\omega\sup_{\tilde x\in \gamma(x_b,x)}
\bigl|
X_n(\tilde x,\epsilon)
-A_n(\tilde x,\epsilon)\bigr|\, | \gamma(x_b,x)|.
$$	
The proposition follows directly from Theorem~\ref{Thm:AnXn}.
\end{proof}

\begin{corollary} \label{Co:Deltahn}
For any compact $D_0 \subset D$ and $\forall x \in D_0$, one has
	\[ h_n(F_\epsilon(x),\epsilon)-h_n(x,\epsilon)=O(\epsilon^{2n}). \]
\end{corollary}

\begin{proof}
Since $F_\epsilon-\Phi^\epsilon_{A_n}=O(|\epsilon|^{2n+1})$, we get
$$
h_n(F_\epsilon(x),\epsilon)=h_n^a(F_\epsilon(x),\epsilon)+O(\epsilon^{2n})=h_n^a(\Phi^\epsilon_{A_n}(x),\epsilon)+O(\epsilon^{2n}).
$$                                                                                                                                                                                                                                                                                                                                                                             
The flow of $A_n$ preserves $h_n^a$, in particular $h_n^a(\Phi^\epsilon_{A_n}(x),\epsilon)=h_n^a(x,\epsilon)$, and the desired estimate follows.
\end{proof}

This corollary implies that  $h_n$ is an adiabatic invariant of the map $F_\epsilon$,
as it is approximately preserved for $\epsilon^{-2n}$ iterations of the map,
provided the corresponding segment of the trajectory remains inside $D$.
We expect that in the case of an analytic map, the adiabatic invariant
$h_{n(\epsilon)}$ is preserved over exponentially long times
where $n(\epsilon)\sim |\epsilon|^{-1}$.

\medskip

We note that the function $h_n$ can be as smooth as the
interpolating vector field (provided the paths $\gamma(x_b,x)$ are 
chosen appropriately).  If the domain $D$ is not simply-connected, 
the function $h_n$ can become multivalued if it
does not return to the original value when $x$ makes a round  along a closed
non-contractible loop.


\section{Numerical study of  dynamics using interpolating vector fields\label{Se:numerics}}

\subsection{Two-dimensional area-preserving maps: the Chirikov standard map} \label{APM}

Our first example is the  Chirikov standard map  \cite{Chi79} written in the
form
\begin{equation}\label{stdmap}
M_{\epsilon}:(x,y) \mapsto (\bar{x},\bar{y})= (x+\epsilon \bar{y}, y-\epsilon \sin(x)),
\end{equation}
where $\epsilon $ is a real parameter.  We consider this map on the cylinder
$\mathbb{S}^1 \times \mathbb{R}$.  On every bounded subset of the cylinder, the
map (\ref{stdmap}) is close to identity provided $|\epsilon|$ is small enough.

In our first numerical experiment we take $\epsilon=0.1$,  $n=5$  and  study  the
interpolating vector field $X_n$ defined by equation \eqref{Eq:interpolVF}. We
choose some initial conditions on the cylinder and compare their iterates under
the original map $M_{\epsilon}$ with the iterates under the time-$\epsilon$ map
$\Phi^\epsilon_{X_n}$ corresponding to the interpolating vector field $X_n$.
Fig.~\ref{stdvscamp_0p1}  shows the first  $10^3$ iterates for  200 initial
conditions. Both pictures use the same set of initial conditions.  The
interpolating vector field is integrated up to $\tau=10^3$ with the help of a
Runge-Kutta-Felberg 7-8 integrator with variable step size. There is no
visually perceptible difference between the two images.

\begin{figure}[htb]
\begin{center}
	\includegraphics[width=7cm]{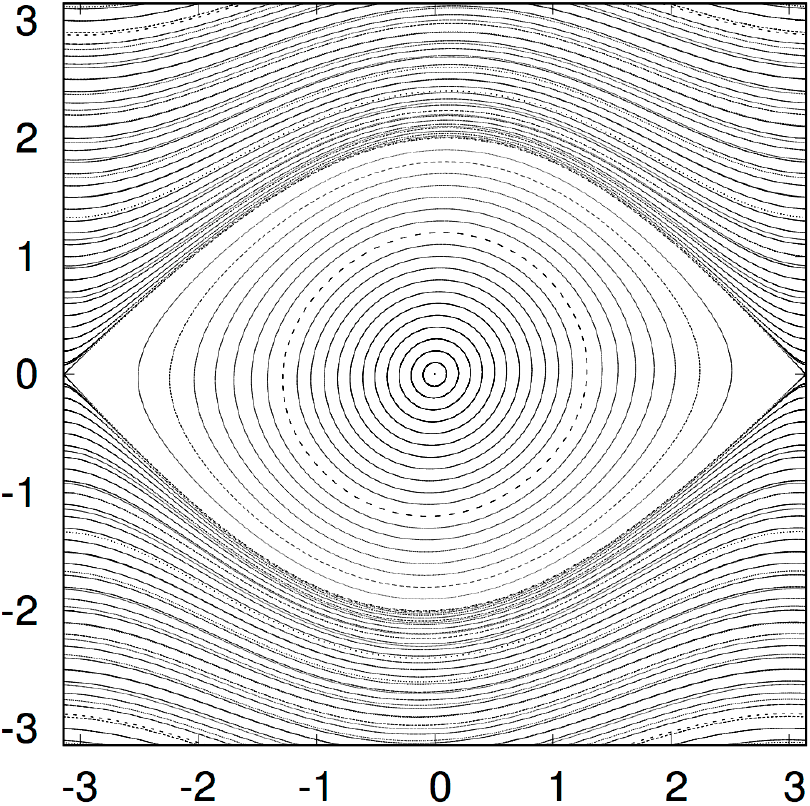}\kern1cm
	\includegraphics[width=7cm]{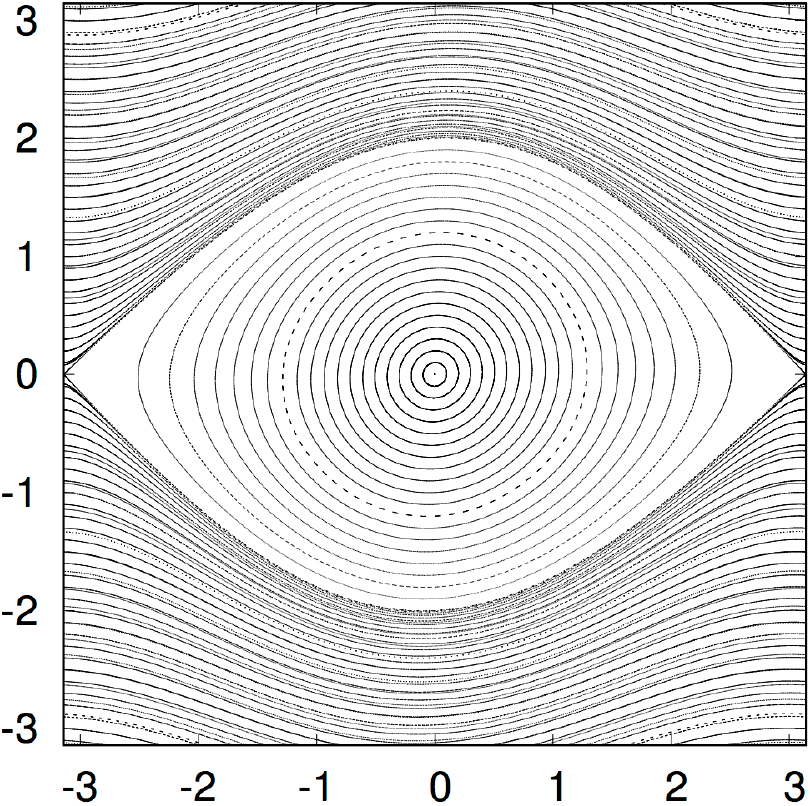} 
\end{center}
	\caption{Iterates of the standard map $M_{\epsilon}$ (left) and  of the 
		time-$\epsilon$ map $\Phi^\epsilon_{X_n}$  associated to the interpolating vector field (with $n=5$) (right). 
		In both cases	$\epsilon=0.1$. No visual differences are detected in the plots. }
	\label{stdvscamp_0p1}
\end{figure}

In order to quantitatively describe the difference between the map and its
interpolating flow, we consider interpolating vector fields $X_n$ for
$n=5,10,15$ and $20$.  Fig.~\ref{stdvscamp_0p1_error} shows level plots of
 $| \Phi^{\epsilon}_{X_n}(x_0)- M_{\epsilon}(x_0)|$ computed for
$5 \times 10^5$ initial conditions selected on a uniform mesh in $[-\pi,\pi] \times [-2 \pi,2\pi]$.  We
clearly see that the error vanishes at the fixed points of the map and that the
error decreases as $n$ increases. Of course,  the interpolation error will
eventually grow with $n$ due to Runge oscillations in interpolation.

\begin{figure}[htb]
	\begin{center}
	\includegraphics[width=7cm]{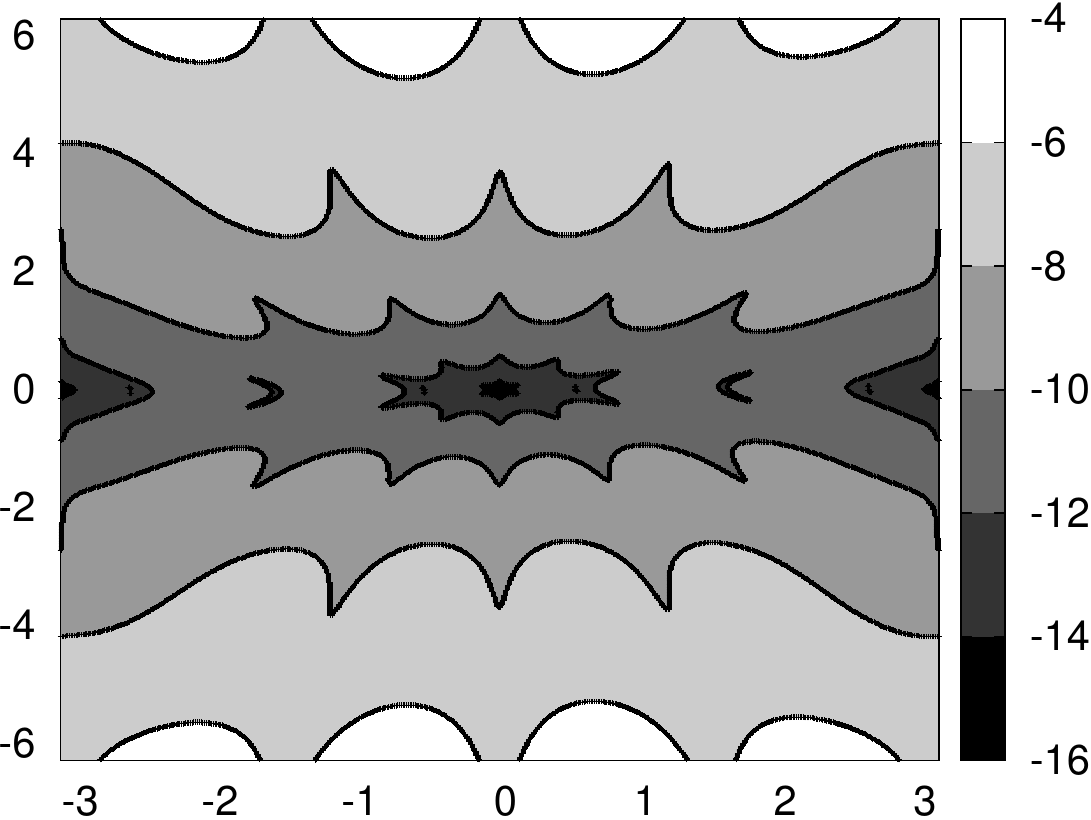} \kern0.5cm
	\includegraphics[width=7cm]{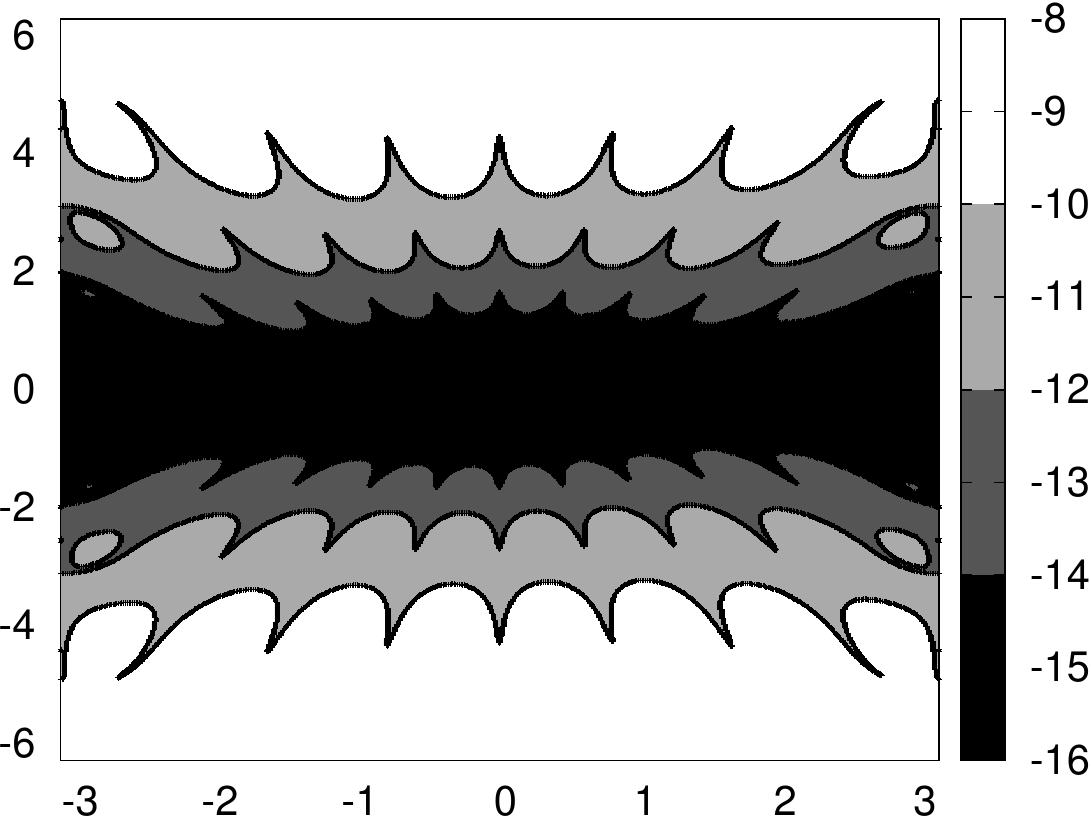}  \\[-5pt]
	\hspace{-0.5cm} $n=5$\hspace{5cm} $n=10$ \\[10pt]
	\includegraphics[width=7cm]{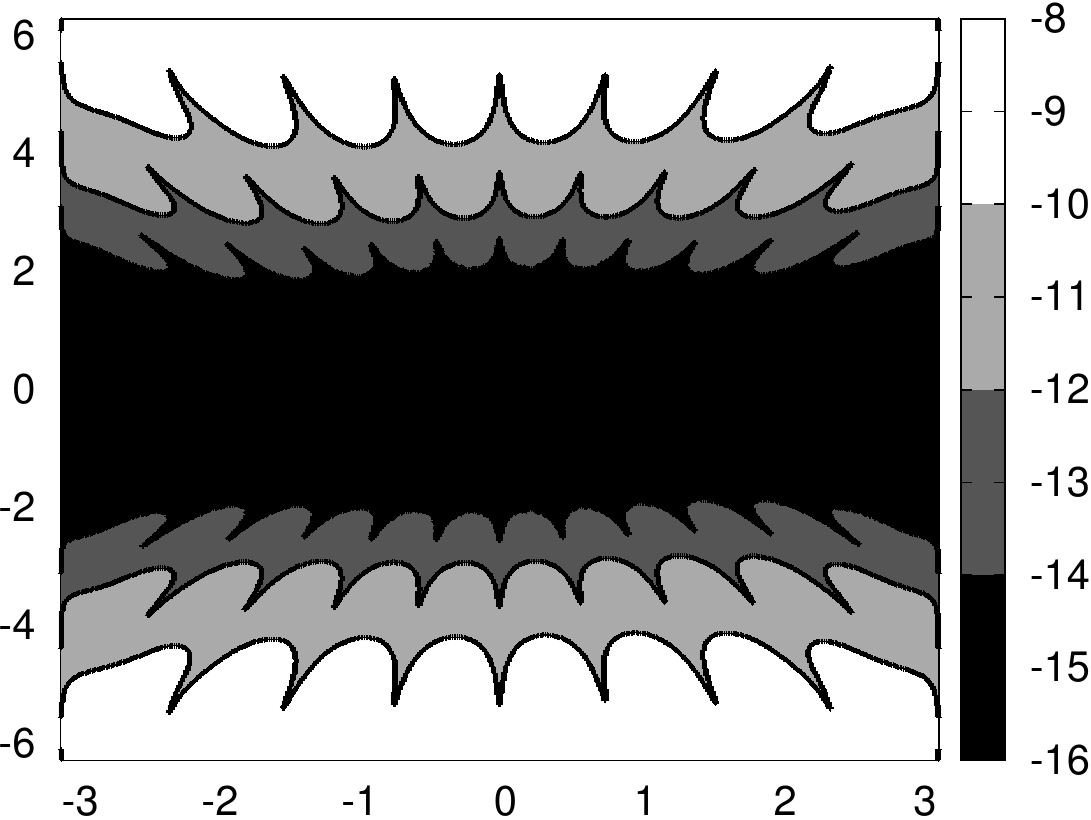} \kern0.5cm
\includegraphics[width=7cm]{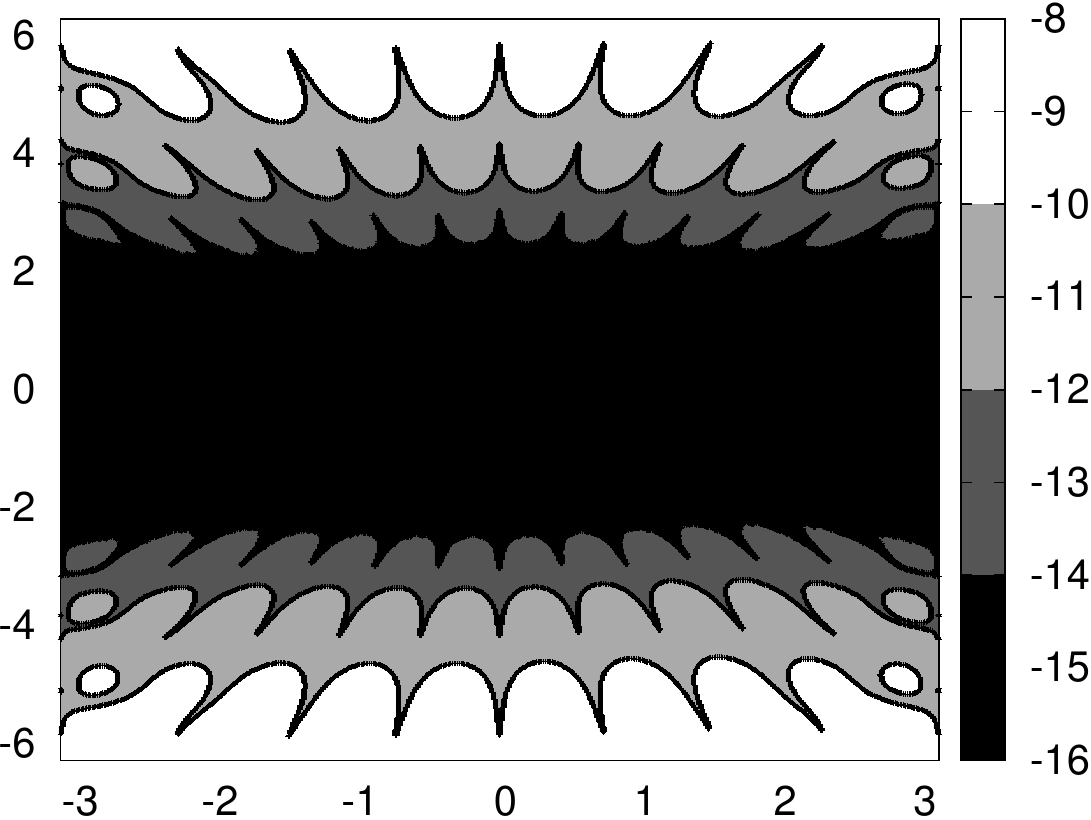}\\[-5pt]
	\hspace{-0.75cm} $n=15$\hspace{4.75cm} $n=20$
	\end{center}
	\caption{$\epsilon=0.1$. Level plots of 
		$\log_{10}|\Phi^\epsilon_{X_n}(x_0)-M_{\epsilon}(x_0)|$ for $x_0\in[-\pi,\pi] \times [-2 \pi, 2 \pi]$,
		  and four different values of $n$.
                }
	\label{stdvscamp_0p1_error}
\end{figure}

\begin{remark} \label{remark_ham2d}
For the standard map \eqref{stdmap}, the
interpolating flow $X_n$, for any $n$, defines a one degree of freedom Hamiltonian system (with
a non-standard symplectic form). This follows from the fact that the map is
reversible, hence so is the interpolating vector field, see Proposition~\ref{Prop:vector_field}.
The reversibility of $X_n$ forces the phase space to be foliated by periodic
orbits, see Fig.~\ref{stdvscamp_0p1} right and also the right plots of
Fig.~\ref{stdvscamp0p5} as illustrations. A reversible 2-dimensional system
having a foliation of periodic orbits is  Hamiltonian (possibly with a
non-standard symplectic  form).
\end{remark}

To visually inspect the differences between the orbits for the map and the
interpolating flow, we increase the parameter up to $\epsilon=0.5$ and show the
results in Fig.~\ref{stdvscamp0p5}.  The left plots represent the iterates of
the standard map $M_{\epsilon}$, while the right ones correspond to the
time-$\epsilon$ map $\Phi^\epsilon_{X_n}$ for $n=10$.  We recall that $X_n$ defines an integrable vector
field, see Remark~\ref{remark_ham2d}. The bottom raw of Fig.~\ref{stdvscamp0p5}
shows  magnifications of a part of the pictures from  the top row: we can
clearly see the differences between the phase portraits when magnifying a strip
located near a  chain of resonant islands of the map. 

\begin{figure}[htb]
\begin{center}
\includegraphics[width=6cm]{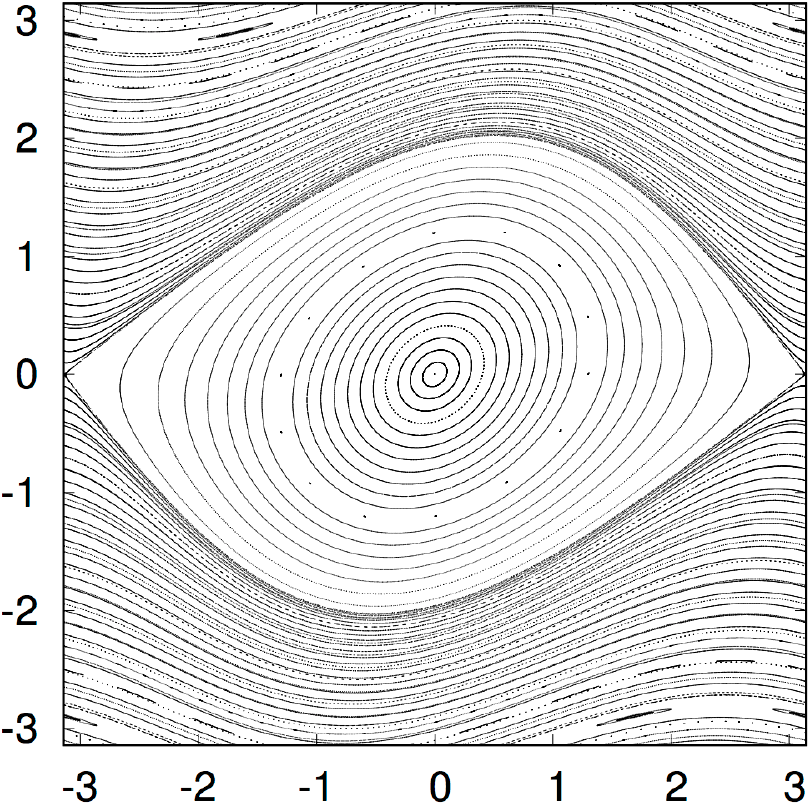} \kern1cm
\includegraphics[width=6cm]{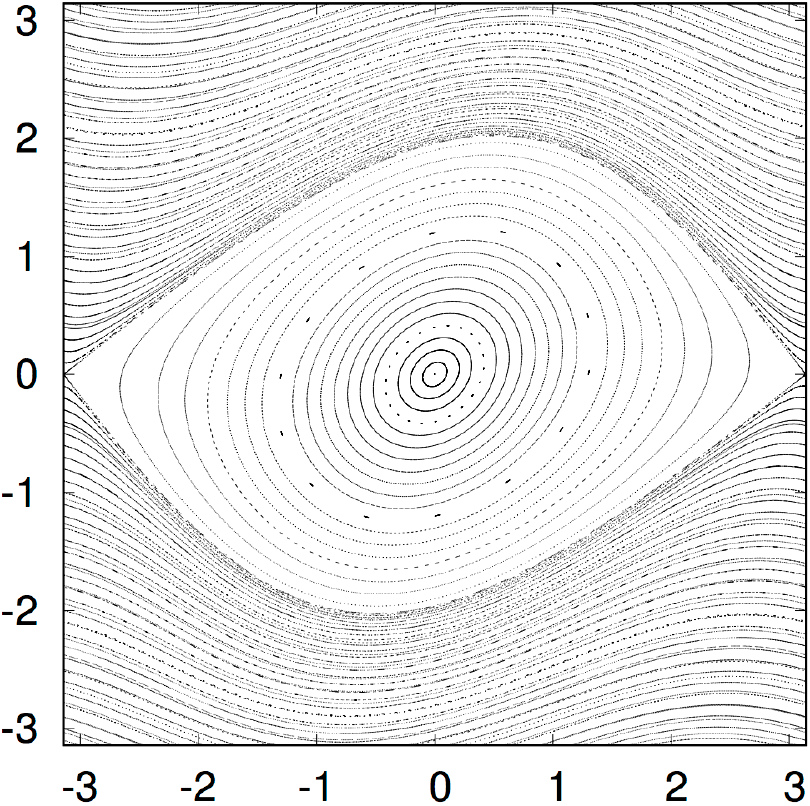} \\[15pt]
\includegraphics[width=6cm]{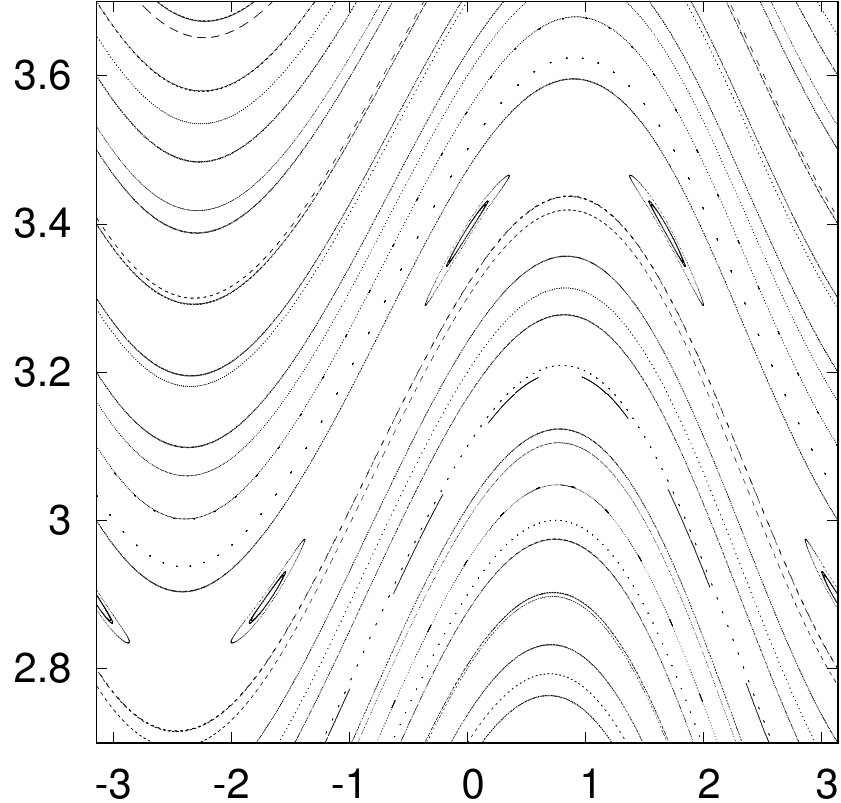} \kern1cm
\includegraphics[width=6cm]{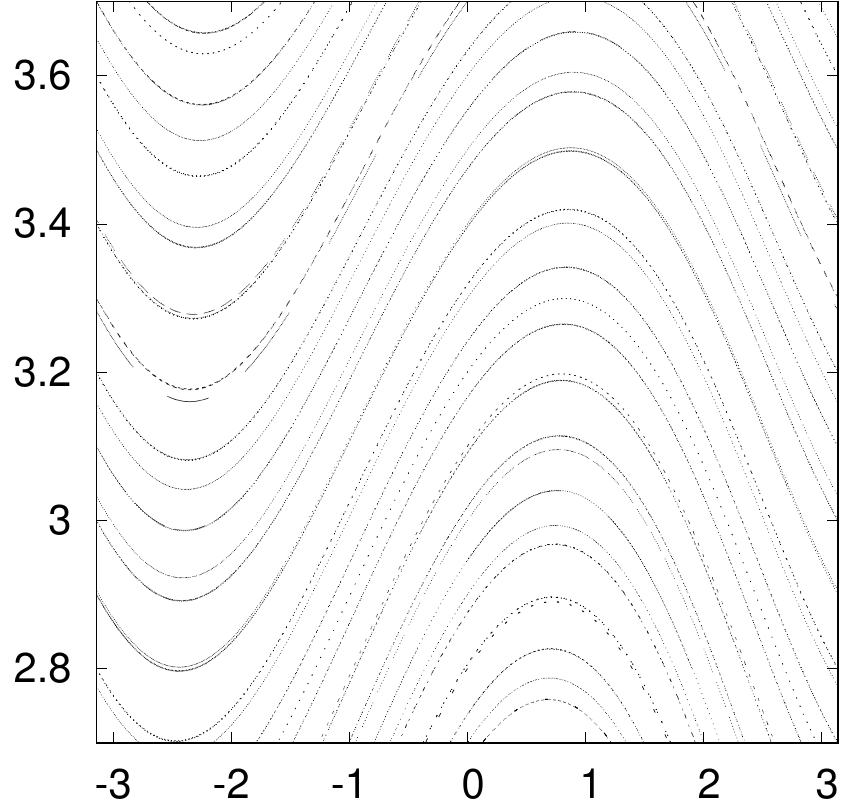} 
\end{center}
\caption{$\epsilon=0.5$. Left: Iterates of the standard map $M_{\epsilon}$. Right: Iterates of the
time-$\epsilon$ map $\Phi^\epsilon_{X_n}$ (with $n=10$) associated to the interpolating vector field. }
\label{stdvscamp0p5}
\end{figure}

To study  the adiabatic invariants defined by the equation \eqref{Eq:hn},
we fix a base point $x_b\in\mathbb R^{2d}$ ($d=1$ in this section but we will also use $d=2$ later) and
 consider paths represented by the function $\gamma(s,x_b,x)=x_b+s v$ where $v=x-x_b$ and $0 \leq s \leq1$.
Then the integral~\eqref{Eq:hn} takes the form
\begin{equation} \label{energy_definition}
h_n(x,\epsilon)=\int_{0}^{1} \sum_{i=1}^d
\left(X_n^i (\gamma(s,x_0,x),\epsilon) v_{i+d}-
X_n^{i+d}(\gamma(s,x_0,x),\epsilon)v_{i}
\right)  ds
\end{equation}
where $v_i$ are components of the vector $v$.  In numerical experiments, we
evaluate  this  integral  using a trapezoidal rule combined with the Romberg
extrapolation scheme. We accept a numerical estimate of the integral value if
the difference between two consecutive approximations of the Romberg scheme is
less than $10^{-8}$.  We use this rule to evaluate the adiabatic invariants in
all examples presented in the paper unless otherwise stated. In principle, this
method can be used to achieve higher precision, however for the visualization
purposes higher precision is not needed.

To investigate the preservation of the adiabatic invariants as a function of
$n$ and~$\epsilon$ under iterates of $M_{\epsilon}$,  we select
$10^4$ initial conditions which form a uniform mesh in $ [-\pi,\pi]^2$ and for every point  compute
$\Delta h_n(x_0)=|h_n(M_{\epsilon}(x_0),\epsilon) - h_n(x_0,\epsilon)|$
and use the corresponding maximum value to estimate the supremum norm  $\|\Delta h_n\|$.

%

\begin{figure}[h]
\begin{center}
\includegraphics[width=7cm]{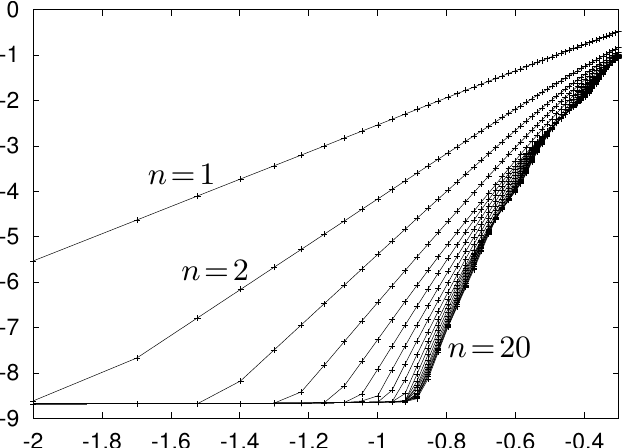} 
\hskip1cm
\includegraphics[width=7cm]{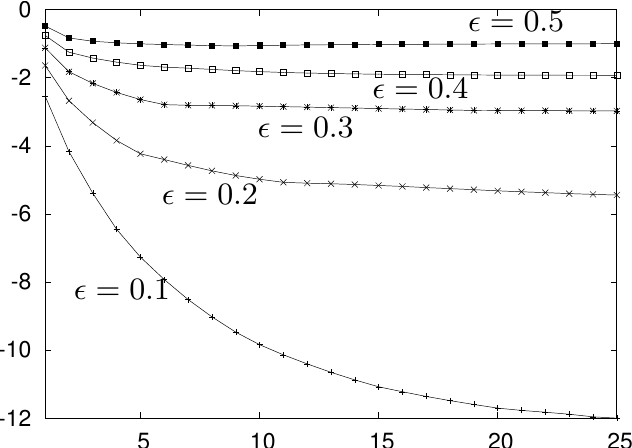} \\
(a)\hskip 7cm(b)
\end{center}
\caption{Plots of  $\log_{10}\|\Delta h_n\|$: (a)  as a function of $\log_{10}\epsilon$ for $1\le n\le 20$, and
	(b)  as a function of $n$ for $\epsilon=0.1(0.1)0.5$.
}
\label{hn_eps}
\end{figure}

Fig.~\ref{hn_eps}(a)  shows  plots of $\|\Delta
h_n\|$ as a function of $\log_{10}\epsilon$ for every $1 \leq n \leq 20$. 
Each line of the plot corresponds to a
fixed value of $n$ and is obtained by joining  $50$ points  with  different values  of $\epsilon \in [0,1/2]$.
 In this plot $n=1$  corresponds to the largest  values of $\|\Delta h_n\|$.
This plot  confirms that  the upper bound 
of  Corollary~\ref{Co:Deltahn} is not too far from being optimal
as the lines on the plot are approximately linear with the slope about $2n+1$ 
(till they reach the levelled floor located just below $10^{-8}$,
which is determined by our choice of the accuracy  in evaluation  of $h_n$).
For $\epsilon\approx0.15$ the values of $\|\Delta h_n\|$
 are monotone decreasing for $1\le n\le 20$.
 On the other hand,  we see that  $\|\Delta h_n\|$
 is not necessarily monotone  in  $n$ for larger values of $\epsilon$ 
 as the lines have intersection. Moreover, the plot indicates that for a fixed $\epsilon$
the value of $\|\Delta h_n\|$  cannot be moved below a certain threshold  by increasing the value of $n$
  (see  Fig~\ref{hn_eps}(b)). Therefore  for a fixed $\epsilon$ we can find an optimal value
  of $n$ which corresponds to a point after which the  adiabatic invariant is not improved when $n$ is increased.
  The existence of this threshold can be attributed to the non-integrability of the map $M_\epsilon$.
  A similar phenomena are observed in the study of  optimal truncation of asymptotic series.

Finally, we note that the methods of  this section can be used to study 
some maps which are not a priori near identity but 
have an iterate which is near identity on some subset of the phase space.
In particular, this situation can arise in a study of a near integrable system near a
multiple resonance.  For example, if
$\epsilon$ is not small, the map $M_{\epsilon}$ is no longer close to identity.
Nevertheless, in a neighbourhood of a $q$-periodic point, the map
$M_{\epsilon}^q$  becomes close to identity.  Let us illustrate how the
interpolating vector field can be adapted to study the dynamics near a
$q$-resonant chain of islands.  Let  $\epsilon=0.5$. We established (see
Fig.~\ref{stdvscamp0p5}) that the interpolating vector field provides an
accurate approximation of the dynamics of $M_\epsilon$ for $y\in[-3,3]$. Now we
consider larger values of $y$ and  investigate what happens with the
approximation.  We take initial points with $x=\pi$ and iterate them. 
Comparing  Fig~\ref{stdvscamp2}~(a) and~(b)
which  represent the dynamics of $M_\epsilon$ and the interpolating vector field $X_n$ for
$n=5$, respectively, we see that the interpolating vector field does not correctly capture the
dynamics around the 2-periodic  chain of islands.
On the other hand,
in this part of the phase space 
the dynamics of $M_{\epsilon}^2$ 
is sufficiently close to identity so that the interpolating vector field $X_{2,n}$,
computed from iterates of $M_{\epsilon}^2$, provides a good approximation of
the dynamics as can be seen in Fig~\ref{stdvscamp2}(c).
Note, that only one of the two islands can be seen due to the choice of initial conditions.

\begin{figure}[htb]
\begin{center}
\includegraphics[width=5.2cm]{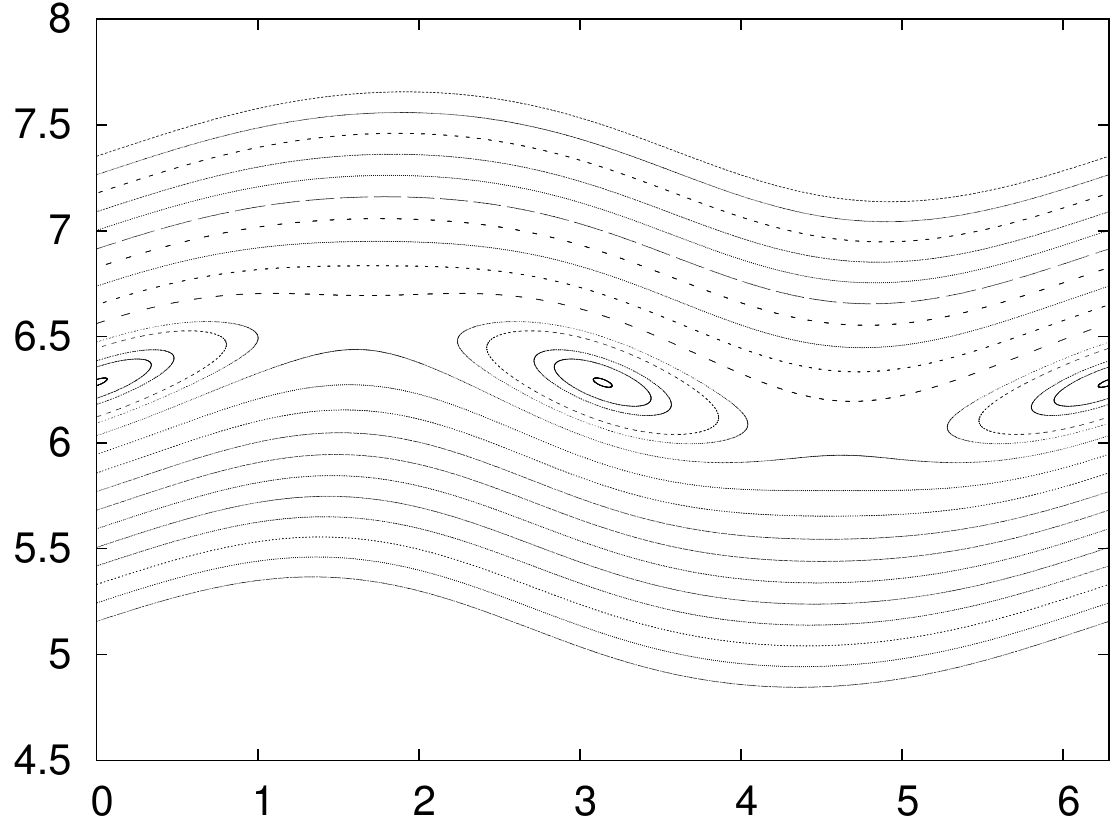}~~
\includegraphics[width=5.2cm]{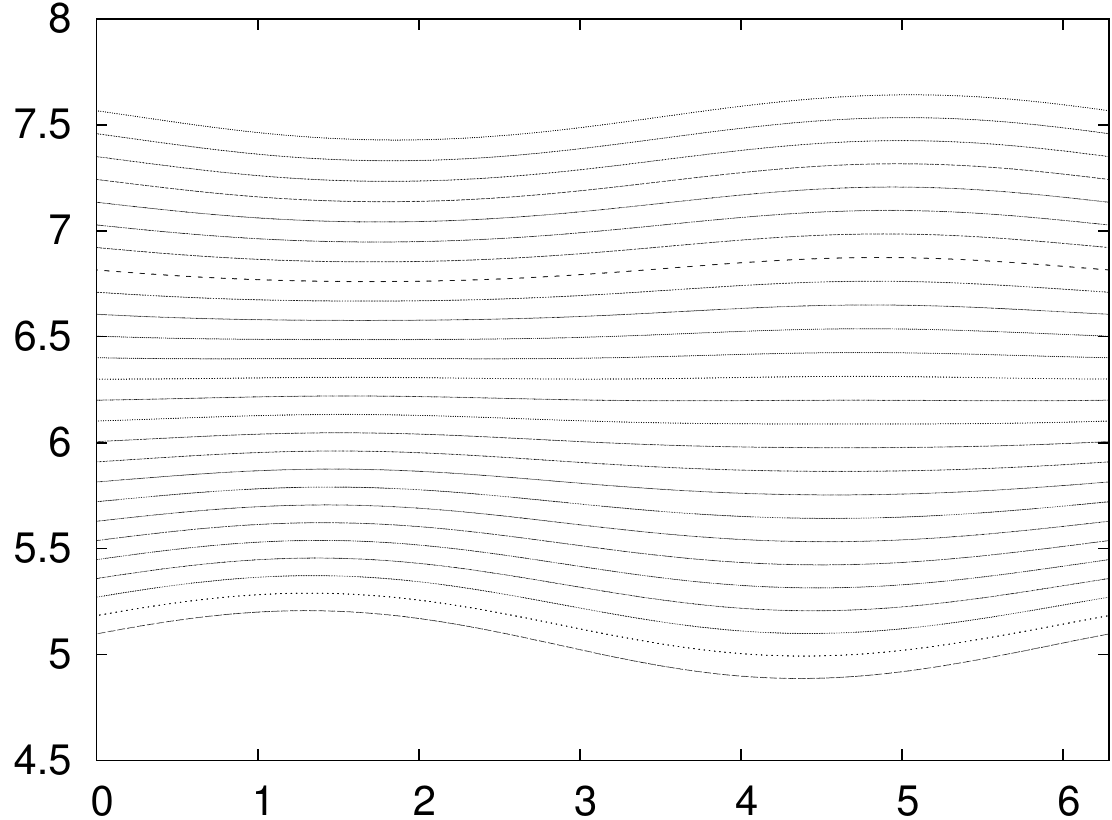}~~
\includegraphics[width=5.2cm]{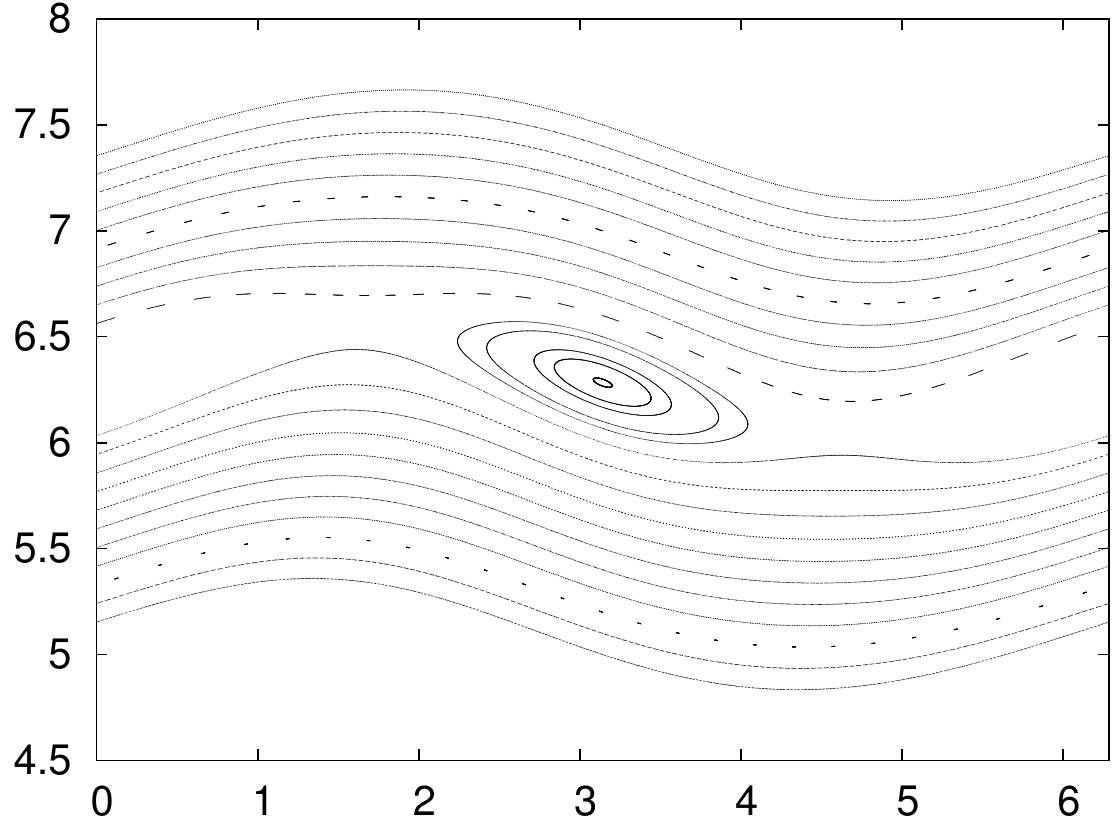}\\
(a)\hskip 5cm (b)\hskip 5cm (c)
\end{center}
\caption{$\epsilon=0.5$. Initial points taken on $x=\pi$, $n=5$. (a)  Iterates of
the standard map $M_{\epsilon}$, (b) Iterates of 
 $\Phi^\epsilon_{X_n}$, and (c)  Iterates of the map
 $\Phi^\epsilon_{X_{2,n}}$  associated with the interpolating
vector field for  $M_{\epsilon}^2$. }
\label{stdvscamp2}
\end{figure}

Fig.~\ref{stdvscamp3} shows similar results for $q=3$. Here we  take initial points with $x=0$,
hence only one of the 3-periodic islands appears in the picture. The interpolating vector
field $X_n^3$ computed for $M_\epsilon^3$  accurately describes  the dynamics in a narrow zone around the resonant
3-periodic island. Note that the 5-periodic island observed in Fig.~\ref{stdvscamp3}(a)
is not present in Fig.~\ref{stdvscamp3}(b).

\begin{figure}[htb]
\begin{center}
\includegraphics[width=7cm]{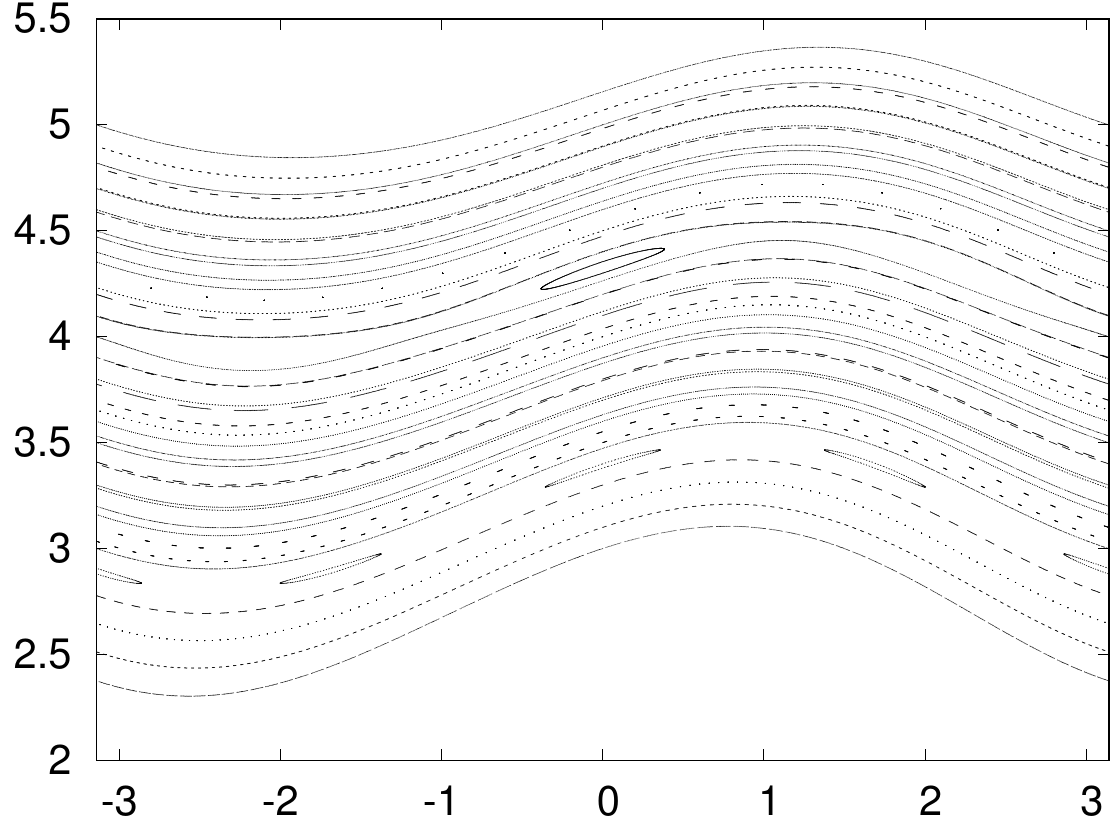}\kern1cm
\includegraphics[width=7cm]{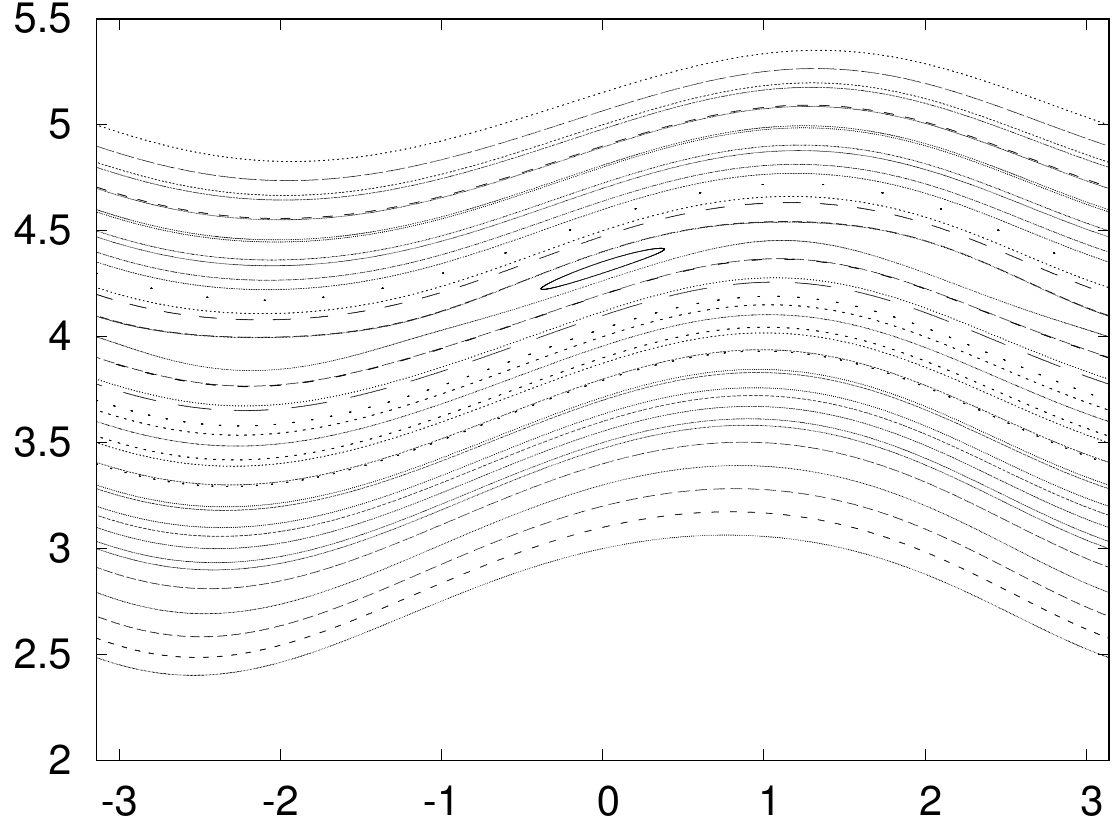}\\
(a)\hskip7cm(b)
\end{center}
\caption{$\epsilon=0.5$.  Initial points taken on $x=0$, $n=5$. 
(a) Iterates of the standard map,  and
(b) iterates 	of  $\Phi^\epsilon_{X_{3,n}}$  associated to the interpolating vector
field for  $M_{\epsilon}^3$. }
\label{stdvscamp3}
\end{figure}

\subsection{Exploring higher dimensional phase spaces: Poincar\'e sections for maps} \label{Poincare}

If the dimension of the phase space $m\ge3$, the visualization of the
dynamics becomes more difficult.  In the case of a system with continuous time, 
a Poincar\'e section provides a convenient tool to reduce the dimension:  a
trajectory  is represented by its intersections with a codimension one surface.  In the case of discrete
time,  the reduction of the 
dimension cannot be performed in a similar way.  A typical solution to this
problem is either to plot a projection of the trajectory on a subset  of lower
dimension, or to use the method of slices \cite{RLBK2014}, i.e.  to plot only a part of the
trajectory which consists of points from a narrow strip near  a codimension one
surface (called a slice).  In the last case,  the points are also projected on the surface
to reduce the dimension.

The interpolating vector fields provide a new tool for visualization of
dynamics which is especially effective in dimensions three and four.  Suppose
that $g:\mathbb R^m\to\mathbb R$ is a smooth function such that its zero set
$\Sigma=\{x\in\mathbb R^m: g(x)=0\}$ is a smooth hyper-surface of codimension
one.  Taking an initial condition $x_0\in D$ we compute the points $x_{k+1}=F_\epsilon(x_{k})$ recursively. 
The  surface $\Sigma$ locally divides  the space. So we can  look for consecutive  points of the trajectory  which
are separated by  $\Sigma$, i.e.,
\begin{equation}\label{Eq:gg}
g(x_k)g(x_{k+1})\le 0.
\end{equation}
Suppose that the limit vector field $G_0$ is transversal to $\Sigma$ 
(at least in a neighbourhood of the intersection of $\Sigma$  with
the straight segment with endpoints at  $x_k$ and $x_{k+1}$).
Then, for $\epsilon$ small enough, there is a unique $t_k\in[0,\epsilon]$ such
that $g(\Phi_{X_n}^{t_k}(x_k))=0$, and we  plot the point 
\begin{equation}\label{Eq:yk}
y_k=\Phi_{X_n}^{t_k}(x_k)
\end{equation}
instead of $x_k$. Obviously, $y_k\in\Sigma$ and the trajectory is represented on
a manifold of lower dimension. The point $y_k$ is $O(\epsilon)$ close to $x_k$
and the error does not accumulate when $k$ grows as the construction of $y_k$ does not
affect the computation of the trajectory $(x_k)_{k\ge0}$.

\subsection{Four-dimensional symplectic maps: the Froechl\'e-like map} 

 We  apply the construction of the previous section  to a 4-dimensional symplectic map
  and show that the method reveals interesting details of the dynamics.
We  consider the Froechl\'e-like map
\begin{equation} \label{4Dmap}
T_{\epsilon}: \left( \begin{array}{c}
           \psi_1 \\
           \psi_2 \\
            J_1 \\
            J_2 
          \end{array} \right) 
\mapsto
\left( \begin{array}{c}
        \bar{\psi}_1 \\
        \bar{\psi}_2 \\
        \bar{J}_1 \\
        \bar{J}_2 
      \end{array} \right) =
\left( \begin{array}{c}
       \psi_1 + \epsilon (a_1\bar{J}_1 + a_2 \bar{J}_2 ) \\
       \psi_2 + \epsilon (a_2 \bar{J}_1 + a_3 \bar{J}_2 ) \\
       J_1 - \epsilon \sin(\psi_1) \\
       J_2 - \epsilon \eta \sin(\psi_2) 
      \end{array} \right),
\end{equation}
where $a_1,a_2,a_3,\eta,\epsilon$ are real parameters.  
 The map $T_\epsilon$ is a symplectic diffeomorphism 
of the cylinder $M=\mathbb T^2\times\mathbb R^2$.
It was introduced in~\cite{GSV2013} to  model  the dynamics
near a double resonance in a near-integrable Hamiltonian system with three
degrees of freedom. In our numerical experiments we  use
\begin{equation} \label{parameters}
a_1=1, \qquad a_2=1/2, \qquad a_3=5/4,\qquad \eta=1/2.
\end{equation}
The quadratic form $a_1 J_1^2+2a_2 J_1J_2+a_3 J_2^2$ is positive definite,
 since  $a_3-a_2^2>0$. The map (\ref{4Dmap}) has four fixed points. 
 If $\epsilon$ is positive and not too large, the origin $p_1 =(0,0,0,0)$ is elliptic-elliptic,
$p_2=(0,\pi,0,0)$ and $p_3=(\pi, 0, 0,0)$ are
hyperbolic-elliptic and $p_4=(\pi,\pi,0,0)$ is hyperbolic-hyperbolic.

\subsubsection*{A  Poincar\'e section for $T_{\epsilon}$}

Since the map \eqref{4Dmap} is symplectic, its limit flow is Hamiltonian.
It is easy to find  the corresponding Hamiltonian function
explicitly:
\begin{equation}\label{Hamiltonian}
h_0(\psi_1,\psi_2,J_1,J_2) =a_1 \frac{J_1^2}{2} + a_2 J_1 J_2 + a_3 \frac{J_2^2}{2}  
- \cos(\psi_1)- \eta \cos(\psi_2).
\end{equation}
This Hamiltonian defines a non-integrable Hamiltonian system with two degrees
of freedom.  The Hamiltonian $h_0$ has four critical points which coincide with
the fixed points of $T_\epsilon$.  Levels of constant energy,
$M_E^0=\{x:h_0(x)=E\}$, are smooth hyper-surfaces for every regular value of
$h_0$. It is natural to study the dynamics of the limit system restricted on each energy
level separately as the Hamiltonian function $h_0$ remains constant along
trajectories of the limit flow.
Let $\Sigma$ be the 3-dimensional hyper-surface  defined by the equality $\psi_1=\psi_2$.
Note that we do not call $\Sigma$ a ``hyper-plane" because we  treat $\psi_1$ and $\psi_2$
as angular variables. The limit vector field is transversal to $\Sigma$ except for
 points which satisfy the equation  $(a_1-a_2) J_1=(a_3-a_2) J_2$ where the vector field is tangent to $\Sigma$. 

The intersection $\Sigma^0_{E}=\Sigma \cap M_E^0$ defines a Poincar\'e section for the limit flow. 
 Outside a neighbourhood of the tangencies, the first return map of the limit flow
 defines an area-preserving map on $\Sigma^0_E$.  The dynamics of the limit
Hamiltonian system are described  by a collection of the  Poincar\'e sections for
different values of~$E$.
Since the quadratic form $a_1 \frac{J_1^2}{2} + a_2 J_1 J_2 + a_3 \frac{J_2^2}{2}  $ is positive
definite,  $M_E^0$ is diffeomorphic to a three dimensional torus $\mathbb T^3$ 
for every  $E>h_0(p_4)=1+\eta$. Then $\Sigma_E^0$ is diffeomorphic to a two dimensional
torus $\mathbb T^2$. It is convenient to use  $\psi=\psi_1=\psi_2$
and $\phi=\arg(J_1+i J_2)$ as  coordinates  on $\Sigma_E^0$.

\medskip

In contrast to the limit flow, the map $T_\epsilon$ does not have a first integral.
Moreover,  even if $x_0\in\Sigma$, it is unlikely that $x_k=T_\epsilon^k(x)$
will ever come back to $\Sigma$ (periodic points of $T_\epsilon$ are obvious exceptions) 
so the direct implementation of the Poincar\'e section is not possible.
In order to visualise a trajectory $x_k=T_\epsilon^k(x_0)$ we implement
the procedure explained in Section~\ref{Poincare}.
The procedure consists in finding a subsequence  $k_j$ such that
the trajectory jumps over $\Sigma$ between $x_{k_j}$ and $x_{k_j+1}$.
We note that $\Sigma$ does not divide the cylinder $M$ into two subsets globally
but, since the map $T_\epsilon$ is near identity, we can check this condition locally.
Since $\Sigma$ is defined by the equality $\psi_1=\psi_2$, we look for $k_j$ such that $(\psi_1^{k_j} -
\psi_2^{k_j}) (\psi_1^{k_{j}+1} - \psi_2^{k_{j}+1}) <0$, where $x_k=(\psi_1^k,\psi_2^k,J_1^k,J_2^k)$
denotes the $k$-th iterate of $x_0 \in \mathbb{T}^2 \times \mathbb{R}^2$.
Then a point $y_{k_j}\in\Sigma$ is defined by projecting $x_{k_j}$ to $\Sigma$
 along the interpolating vector field $X_n$ as defined by the equation  \eqref{Eq:yk}.

Since the section $\Sigma$ is three dimensional, the sequence $y_{k_j}$
can be plotted and used to visually inspect the behaviour of the trajectory $x_k$.

\medskip

On a moderate time scale, a further reduction of the dimension can be achieved by noting that $h_n$ 
(defined by the equation~\eqref{energy_definition}) is
an adiabatic invariant of the map, so the trajectory $x_k$ stays in a small neighbourhood
of the set $M_E^n=\{x:h_n(x,\epsilon)=E\}$ where $E=h_n(x_0,\epsilon)$.
Since $M_E^n$ is close to $M_E^0$ and the latter is nicely 
described by the coordinates $(\psi,\phi)$, we can project the points $y_{k_j}$
on the torus of the  coordinates $(\psi,\phi)$. In this way a trajectory
of a 4-dimensional map is represented by a sequence of points on a two dimensional
torus.

\medskip

We remark that this procedure relies on the closeness of the map  to the identity.
Similar to the standard map,  acceptable values of $\epsilon$ depend
on the values of the variables $J_1$ and $J_2$. In our numerical experiments we use
$|\epsilon|\le0.5$, so in practical terms the parameter does not need to be very small.

\subsubsection*{Visualization of the dynamics of $T_{\epsilon}$}

Examples of visualization of dynamics for $T_\epsilon$ are shown on Fig.~\ref{3dPoinc} and~\ref{3dPoinc0p2}.
Some comments concerning the implementation can help the reader. First,
the computation of the points $y^{k_j}$ on $\Sigma$ requires
integration of $X_n$ which is performed using a RK7-8 method that only requires
evaluating the vector field. The time $t_k$ in \eqref{Eq:yk} is then computed using the Newton
method in a way similar to   \cite{Sim89}.

Second, to  show different trajectories on a single 2-dimensional torus
we select initial conditions on $\Sigma_E^n=\Sigma \cap M_E^n$ for a fixed $E$.  
To find initial conditions with the same value of $E$, we use the following procedure:
we select values of $\psi=0,1,2,3$ and, for each value, 
we compute a point $p=(\psi,\psi,0,J_2^0)$, with $J_2^0>0$, such
that $h_n(p,\epsilon)=E$ (using a bisection method with respect to $J_2$ to get a zero of
$h_n(p,\epsilon)-E$).  Since $\nabla h_n$ is orthogonal to the vector $(0,0,-\partial h_n / \partial
 J_2, \partial h_n / \partial J_1)$, we numerically integrate the auxiliary
vector field
\begin{equation} \label{2dauxiliar}
\dot{J}_1 = -\frac{\partial{h_n}}{\partial J_2} \, ,\qquad
\dot{J}_2 = \frac{\partial{h_n}}{\partial J_1} \, ,
\end{equation}
with initial condition $(J_1(0),J_2(0))=(0,J_2)$ (using a RK7-8 method).
One obtains points $x_{0,i}=(\psi,\psi,J_1(t_i),J_2(t_i)) \in \Sigma^n_E$ for a
sequence of $t_i$ provided by the integration method. 
Finally, we use $x_{0,i}$ as initial conditions for the map~$T_\epsilon$.

\begin{figure}[thbp]
	\begin{center}
		\includegraphics[width=16cm]{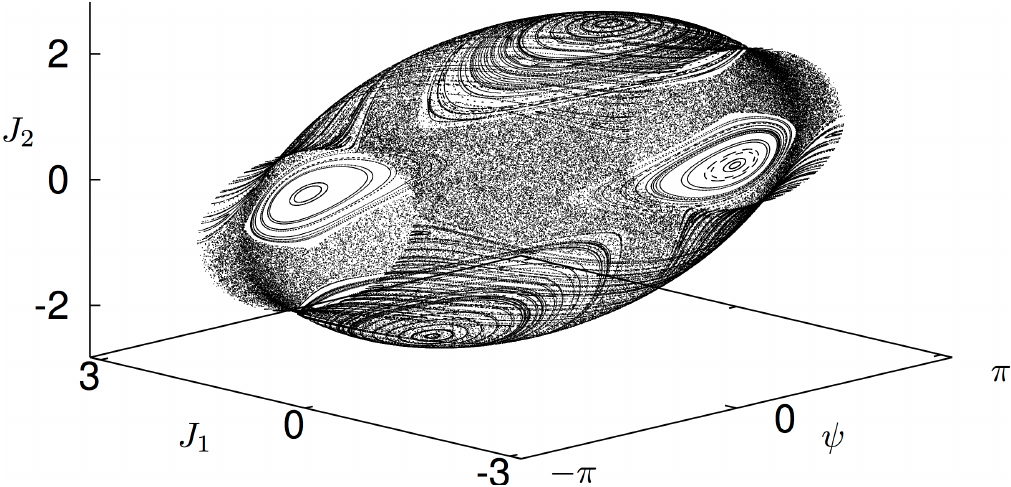}  \\[10pt]
                (a)\\[15pt]
		\includegraphics[width=8cm]{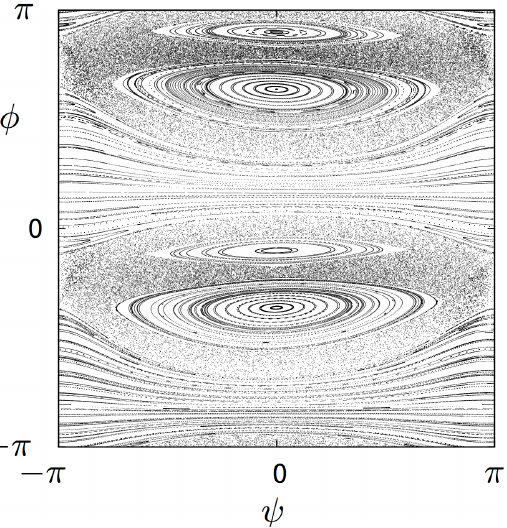}\kern 0.5cm
		\includegraphics[width=8cm]{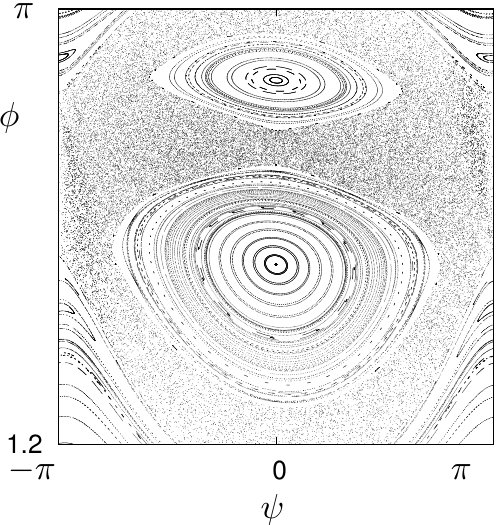}
                (b)\kern 7.25cm(c)
	\end{center}
	\caption{Trajectories of $T_\epsilon$ with the parameters defined by (\ref{parameters}) and $\epsilon=0.2$, $E=1$, $n=10$.
		(a) $500$ projections $y_{k_j}$ are plotted for each of $400$ initial conditions taken
		on $\Sigma \cap M^n_E$, (b) projection of the points of (a) onto the torus with coordinates $(\psi,\phi)$,
        where $\phi = \arg(J_1+ i J_2) \in (-\pi,\pi]$, (c) is a magnification from (b).
        }
	\label{3dPoinc0p2}
\end{figure}

Fig.~\ref{3dPoinc0p2} shows 500 projected iterates $y_{k_j}$, as defined by
\eqref{Eq:yk}, obtained from the iterates $x_{k_j}$ under the map $T_{\epsilon}$ for
each of around 400 different initial conditions.  The parameters of the map are defined by
(\ref{parameters}) and $\epsilon=0.2$.  The
hyperbolic-hyperbolic fixed point $p_4$ is used as a base point $x_b$ in the
definition of the adiabatic invariant and the initial conditions are chosen on
$M_E^n$ with  $n=10$ and $E=1$. We see that all points are located near a
2-dimensional torus  embedded in the 3-dimensional surface $\Sigma$.  The
projection of the trajectories onto the coordinates $(\psi,\phi)$ resembles the
dynamics of an area-preserving map.

\begin{figure}[thbp]
	\begin{center}
		\includegraphics[width=7cm]{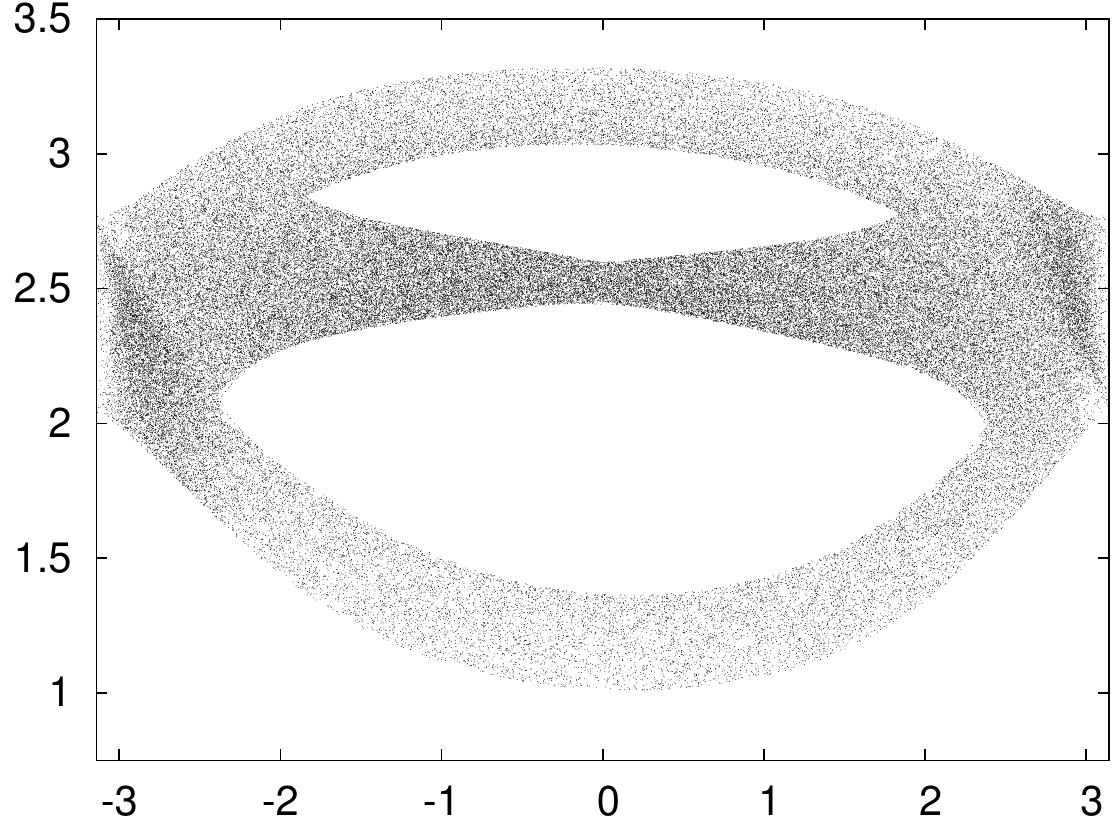}\kern1.5cm
		\includegraphics[width=7cm]{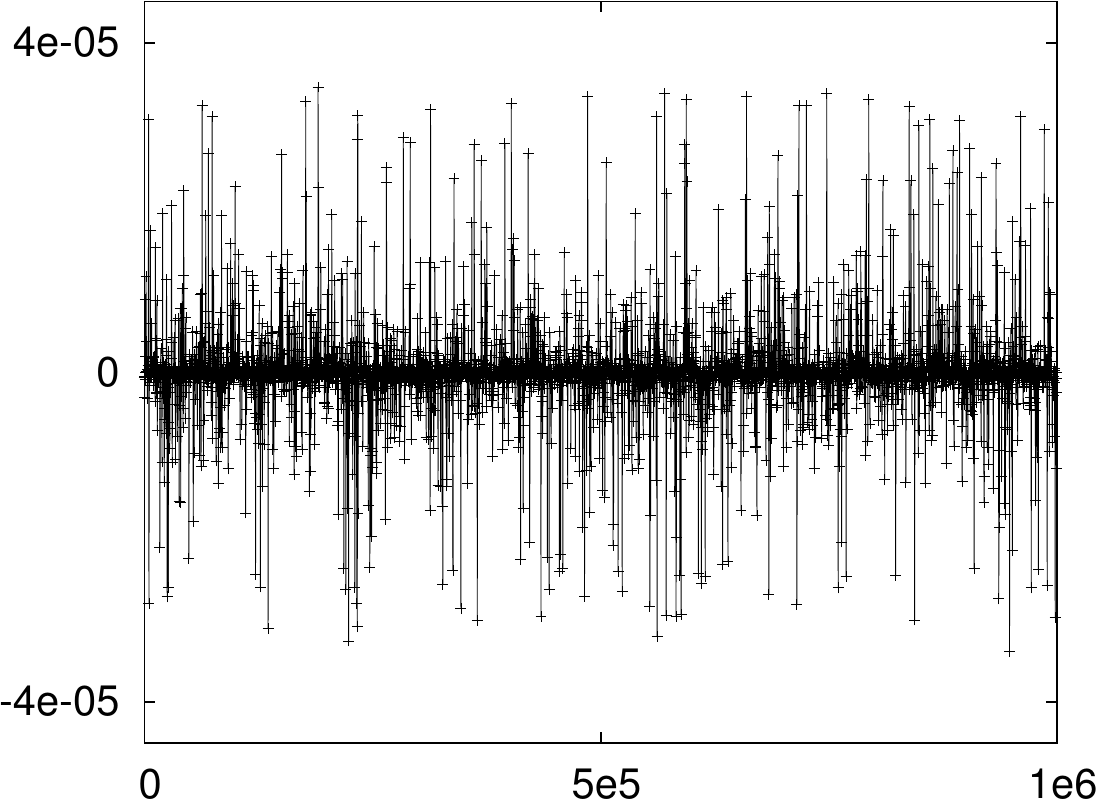}\\
		\hspace{0.5cm} (a)\kern7cm(b)\\[12pt]
		\includegraphics[width=7cm]{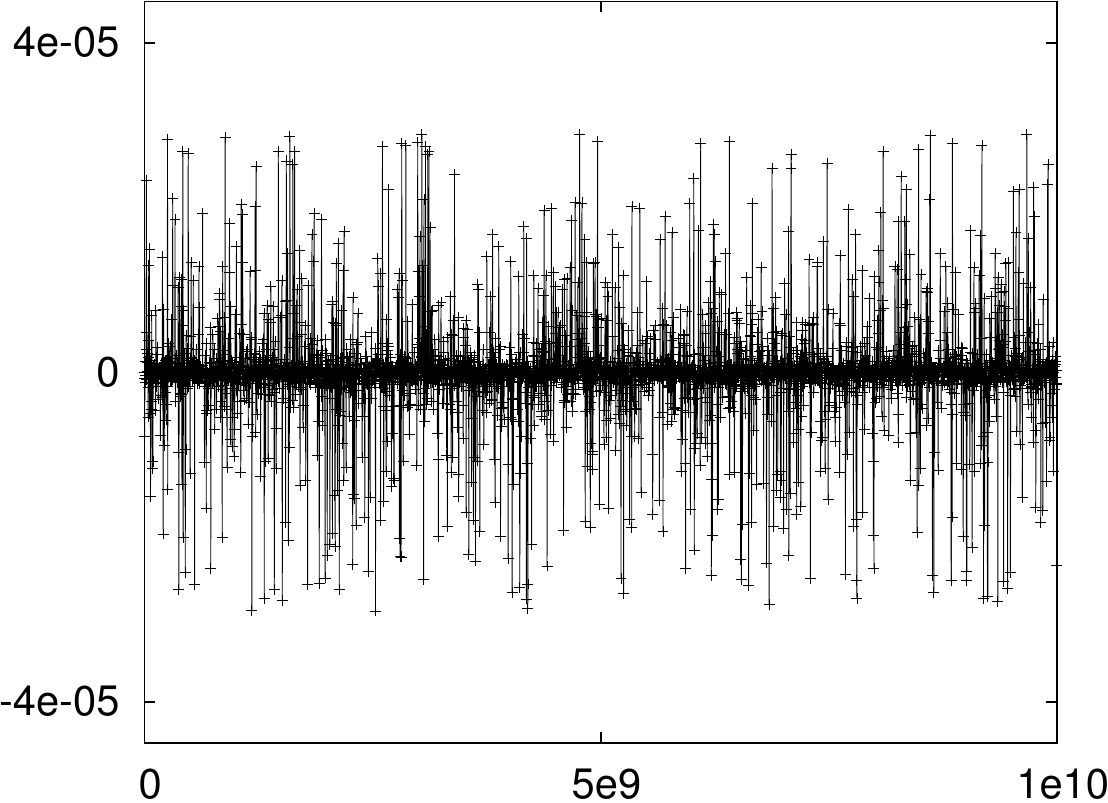}\kern1.5cm
		\includegraphics[width=7cm]{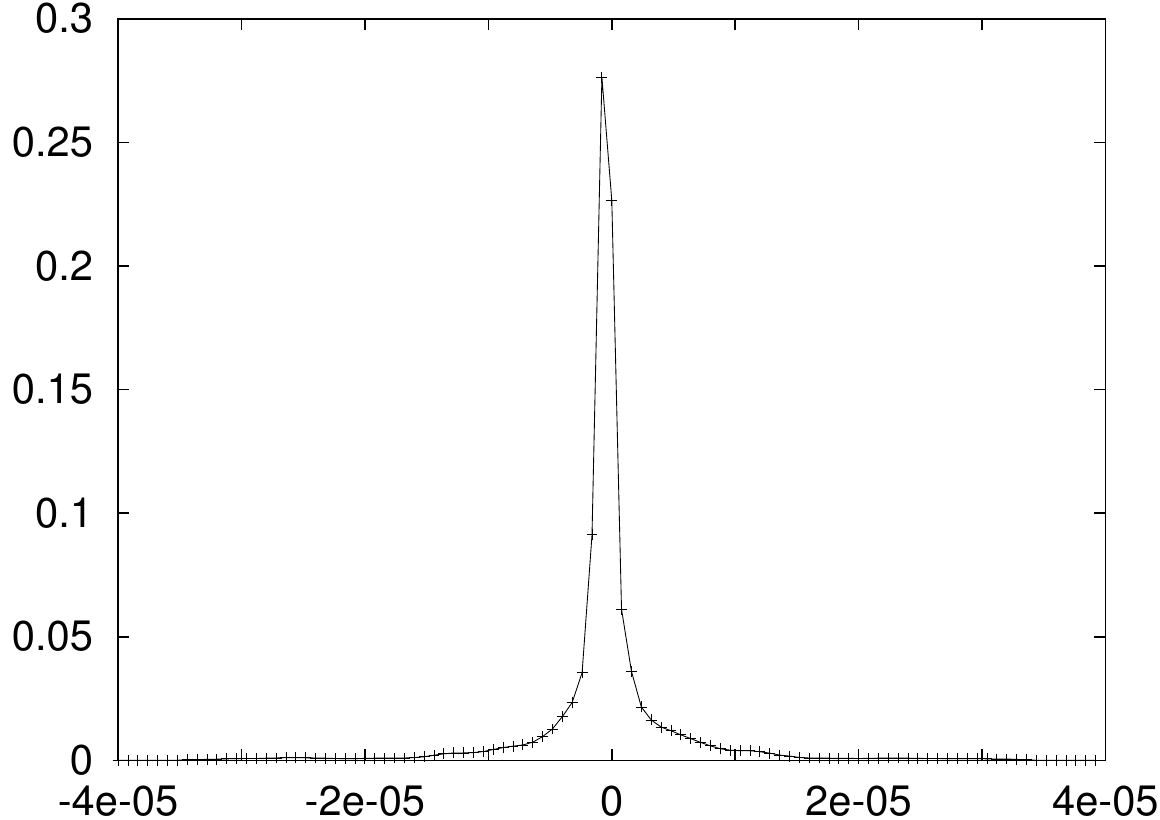}\\
		\hspace{0.5cm} (c)\kern 6.5cm(d)
	\end{center}
	\caption{ A single trajectory of the map $T_\epsilon$ with $\epsilon=0.2$
		for the initial point located at $(3, 3, -1.043523, 1.385456)$. 
		(a) Points on the Poincar\'e section are represented using coordinates 
		 $\psi_1\ (=\psi_2)$ and $\phi = \arg(J_1+ i J_2) \in (-\pi,\pi]$;
	          we show only one out of every ten consecutive points on the Poincar\'e section.
(b) The adiabatic invariant $h_n$ as a function of time (we show one point out of every 250 iterates).
(c)  The same as in (b)  but up to $10^{10}$ iterates instead (we show one point out of every $250 \times 10^4$ iterates). 
(d)  A histogram which illustrates the distribution of the adiabatic invariant
for the Poincar\'e iterates up to $t=10^6$. The values of $E$ for the $10^6$
iterates considered range in $[-10^{4},10^{4}]$ and a bin size of $8\times
10^{-7}$ has been used to obtain the distribution.
	}
	\label{pres_point}
\end{figure}

Fig.~\ref{pres_point}(a)  shows  iterates
of the point $(\psi_1,\psi_2,J_1,J_2)=(3, 3, -1.043523, 1.385456)$ that belongs to $\Sigma_E^n$ with $E=1$.
This is one of the chaotic orbits which can be seen in Fig.~\ref{3dPoinc0p2}.
Fig.~\ref{pres_point}(b) shows a plot  of the adiabatic invariant as a function of time
along this orbit. For
better visualization we show the value of the adiabatic invariant for one
out of every 250 consecutive iterates of the point under $T_{\epsilon}$. Some important observations
follow from  this plot:
\begin{itemize}
\item The orbit we study apparently belongs to a chaotic zone and is not
located on a KAM torus.	

\item The adiabatic invariant is preserved (meaning that the range of
oscillations remains unchanged) up to, at least, $10^6$ iterates.
Fig.~\ref{pres_point}(c) shows that this behaviour continues  at
least for  $10^{10}$ iterates of $T_\epsilon$. The dynamical interpretation of such
preservation is clear: our numerical method has detected the slowest variable of the
system which evolves on a very long  time scale. We note that in this example
one expects to see an (Arnold like) diffusion process and, in particular,  the
slow variable requires exponentially long (with respect to $\epsilon^{-1}$)
time to be changed by order one. Our method is able to detect such a slow
variable in a simple and efficient way.

\item There is no systematic drift of the adiabatic invariant.
Its values are distributed in a Gaussian-like way around the initial value
(see Fig.~\ref{pres_point}(d)). We remark that, for $T_{\epsilon}$
and the chosen parameters, there is numerical evidence supporting that the
detection of the expected drift in the adiabatic invariant due to the Arnold
diffusion process requires a much longer time scaling.
\end{itemize}

A final illustration of the amount of details one can visualize with this
methodology is presented in  Fig.~\ref{3dPoinc}. 
There we show a plot similar to the one in
Fig.~\ref{3dPoinc0p2} but for $\epsilon=0.35$ and $E=4$.  As before, we take
the hyperbolic-hyperbolic fixed point $p_4$ as a base point to define the
adiabatic invariant.  The interpolating vector field is constructed with
$n=10$. One can clearly recognize the typical structures that show up in 
a phase
space of an area-preserving map: islands of stability (including secondary
islands and satellites in the chaotic zone), invariant curves and chaotic zones.
Yet this is not a 2-dimensional map: we are just plotting a suitable
projection (along the solutions of the interpolating vector field) of the
iterates of the 4-dimensional map!

\section{Acknowledgments}

AV has been supported by grants MTM2016-80117-P (Spain) and 2014-SGR-1145 (Catalonia). 
 VG's research was supported by EPRC (grant EP/J003948/1).
 The authors are grateful to Ernest Fontich and Carles Sim\'o for useful discussions
 and remarks.

\bibliographystyle{alpha}

\begin{thebibliography}{AN75}
\bibitem{BogMit1961}	
	Bogoliubov, N. N., Mitropolsky, Y. A.;
	{\em  Asymptotic methods in the theory of non-linear oscillations}.
	 Gordon and Breach Science Publishers, New York 1961, 537 pp.
	
	
\bibitem{BroRouSim96}
   Broer, H., Roussarie, R., and Sim\'o, C.;
   {\em Invariant circles in the Bogdanov-Takens diffeomorphisms.}
   Ergod. Th. and Dynam. Sys., 16:1147 -- 1172, 1996.


\bibitem{Chi79}
   Chirikov, B.V.;
   {\em A universal instability of many-dimensional oscillator system}.
   Phys. Rep., 52 (1979), 264--379.
   
\bibitem{EH13}
  Efthymiopoulos, C., and Harsoula, M.;
  {\em The speed of Arnold diffusion},
  Physica D vol. {\bf 251} issue 1 (2013) pp. 19--38.   
   
\bibitem{G2002}
   Gelfreich, V.,
{\em   Numerics and exponential smallness}. In Handbook of dynamical systems, Vol. 2, 265--312, North-Holland, Amsterdam, 2002. 

\bibitem{GSV2013}
   Gelfreich, V., Sim\'o, C., and Vieiro A.;
   {\em Dynamics of 4D symplectic maps near a double resonance},
   Physica D vol. {\bf 243} issue 1 (2013) pp. 92--110.


\bibitem{Nei84}
   Neishtadt, A.I.;
   {\em The separation of motions in systems with rapidly rotating phase},
   J.Appl. Math. Mech., 48 (1984), 133--139.
   
\bibitem{Nekhoroshev1977}   
Nekhoroshev N.N.;
{\em An exponential estimate of the time of stability of nearly integrable Hamiltonian systems.}
Russian Mathematical Surveys,  32(6) (1977) 1--65.

\bibitem{RLBK2014}   
   Richter M., Lange S., Bäcker A, Ketzmerick R.;
   {\em Visualization and comparison of classical structures and quantum states of four-dimensional maps}.
    Phys. Rev. E 89(2) (2014) 022902
   

\bibitem{Sim89}
 Sim\'o, C.;
 {\em On the Analytical and Numerical Approximation of Invariant Manifolds},
 Les M\'ethodes Modernes de la Mec\'anique C\'eleste (Course given at Goutelas, France, 1989), D. Benest and C. Froeschl\'e (eds.), pp. 285--329, Editions Frontières, Paris, 1990. 

\bibitem{Sim00}
 Sim\'o, C.;
 {\em Analytic and numeric computations of exponentially small phenomena},
 In {K}.~{G}r{\"o}ger {B}.~{F}iedler and {J}. {S}prekels, editors,
 {\em {P}roceedings {EQUADIFF} 99, {B}erlin}, pages 967 -- 976, 2000.
 {W}orld {S}cientific, {S}ingapore.


\bibitem{textbook} Stoer, J; and Bulirsch, R.;
  {\em Introduction to Numerical Analysis}. Third edition.  Texts in Applied
  Mathematics, 12. Springer-Verlag, New York, 2002. xvi+744 pp. 

   
\bibitem{Tak74}
 Takens, F.;
 {\em Forced oscillations and bifurcations},
 Applications of Global Analysis I. Communications of the
 Mathematical Institute Rijksuniversiteit Utrecht, 3, 1974.

\end{thebibliography}

\end{document}